\documentclass[a4paper,reqno, 10pt]{amsart}

\usepackage{amsmath,amssymb,amsfonts,amsthm, mathrsfs}
\usepackage[latin1]{inputenc}
\usepackage{lmodern}

\usepackage{verbatim,wasysym,cite}

\usepackage[latin1]{inputenc}
\usepackage{microtype}
\usepackage{color,enumitem,graphicx}
\usepackage[colorlinks=true,urlcolor=blue, citecolor=red,linkcolor=blue,
linktocpage,pdfpagelabels, bookmarksnumbered,bookmarksopen]{hyperref}
\usepackage[english]{babel}
\usepackage[symbol]{footmisc}
\renewcommand{\epsilon}{{\varepsilon}}

\numberwithin{equation}{section}
\newtheorem{theorem}{Theorem}[section]
\newtheorem{lemma}[theorem]{Lemma}
\newtheorem{remark}[theorem]{Remark}

\newtheorem{proposition}[theorem]{Proposition}

\newcommand{\C}{\mathbb C}

\newcommand{\R}{\mathbb R}
\newcommand{\N}{\mathbb N}

\def\({\left(}
\def\){\right)}
\def\<{\left\langle}
\def\>{\right\rangle}

\def\O{\mathcal O}

\def\K{\mathcal K}

\def\EE{\mathcal E}

\def\eps{\varepsilon}

\DeclareMathOperator{\RE}{Re}
\DeclareMathOperator{\IM}{Im}

\begin{document}

\title[NLS with a repulsive Dirac delta potential]
{Threshold scattering for the focusing NLS with a repulsive Dirac delta potential}

\author[Alex H. Ardila]{Alex H. Ardila}
\address{Department of Mathematics, Universidade Federal de Minas Gerais\\ ICEx-UFMG\\ CEP
30123-970\\ MG, Brazil} 
\email{ardila@impa.br}

\author[Takahisa Inui]{Takahisa Inui}
\address{Department of Mathematics, Graduate School of Science, Osaka University\\ Toyonaka\\ Osaka\\ 560-0043, Japan} 
\email{inui@math.sci.osaka-u.ac.jp}

\begin{abstract}
We establish the scattering of solutions to the focusing mass supercritical nonlinear Schr\"odinger equation
with a repulsive Dirac delta potential
\[
i\partial_{t}u+\partial^{2}_{x}u+\gamma\delta(x)u+|u|^{p-1}u=0, \quad (t,x)\in \R\times\R,
\]
at the mass-energy threshold, namely, when
$E_{\gamma}(u_{0})[M(u_{0})]^{\sigma}=E_{0}(Q)[M(Q)]^{\sigma}$
where $u_{0}\in H^{1}(\R)$ is the initial data, $Q$ is the ground state of the free NLS 
on the real line $\R$, $E_{\gamma}$ is the energy, $M$ is the mass and $\sigma=(p+3)/(p-5)$.
We also prove failure of the uniform space-time bounds at the mass-energy threshold.
\end{abstract}


\maketitle

\medskip

\section{Introduction}
\label{sec:intro}
In this paper we consider the following nonlinear Schr\"odinger equation with a
delta potential
\begin{equation}\label{NLS}
\begin{cases}
i\partial_{t}u(t,x)=-\partial^{2}_{x}u- \gamma\delta(x)u-|u|^{p-1}u,\\
u(0,x)=u_{0}\in H^{1}_{x}(\R),
\end{cases}
\end{equation}
where $\gamma\in (-\infty, 0)$, $p>5$ and  $u=u(t,x)$ is a complex-valued function of $(t,x)\in \mathbb{R}\times\mathbb{R}$.
Here, $\delta$ is the Dirac distribution at the origin, i.e. $\<\delta, v\>=v(0)$ for $v\in H^{1}(\R)$.
The Dirac distribution $\delta$ is used to model a defect localized at the origin (see, for example, \cite{GHW}). 

Due to their abundance physical and mathematical properties the NLS \eqref{NLS} has drawn much attention of physicists and mathematicians \cite{AFI, GHW, APLD, BMQT}. The  effect  of  the $\delta(x)$-potential on the  dynamics  of  the  nonlinear  Schr\"odinger  equation has been studied intensively in later years.  The Cauchy problem,  existence of ground states and their stability/instability, 
long time dynamics (scattering and global existence, blow-up) to \eqref{NLS} with data below the ground state threshold,  etc., have been  studied in recent years; see for example \cite{ArCeGo, RJJ, FO, LFF, MT2, BaniVisci2016} for more details.

The formal expression of the operator $-\partial^{2}_{x}-\gamma\delta(x)$ which appears in \eqref{NLS} admits a precise interpretation as a self-adjoint  operator $H_{\gamma}$ on the space $L^{2}(\R)$. Indeed,  we have formally
\begin{equation*}
\< (-\partial^{2}_{x}-\gamma\delta(x))f, g\>=\mathfrak{t}_{\gamma}[f,g] \quad \text{for $f$, $g\in H^{1}(\R)$},
\end{equation*}
where $\mathfrak{t}_{\gamma}$ is the quadratic form defined by 
\begin{equation}\label{UKI}
\mathfrak{t}_{\gamma}[f,g]=\RE\int_{\mathbb{R}}\partial_{x}f\overline{\partial_{x} g} dx-\gamma\RE \left[f(0) \overline{g(0)}\right].
\end{equation}
As the bilinear form  $\mathfrak{t}_{\gamma}$ is bounded from below and closed on $H^{1}(\R)$, it is possible to show that the self-adjoint operator associated with $\mathfrak{t}_{\gamma}$ is  given by (see \cite[Theorem 10.7 and Example 10.7]{KSN})
\[
\begin{cases}
 H_{\gamma} f(x)=-\frac{d^2}{dx^2} f(x)\quad \text{for}\;\; x\neq 0,\\
f\in \mathrm{dom}(H_{\gamma})=\left\{f\in H^{1}(\R)\cap H^{2}(\R\setminus \left\{0 \right\}) : f'(0+)- f'(0-)=-\gamma f(0)\right\}.
\end{cases}
\]
We observe that the operador  $H_{\gamma}$ can also be defined via theory of self-adjoint extensions of symmetric operators (see \cite{APIN})

It is well-known that (see \cite{FO}) for $u_{0}\in H^{1}({\mathbb{R}})$, there exists $T_{\ast}=T(\|u_{0}\|_{H^{1}})>0$ and a unique maximal solution  $u(t,x)$ to \eqref{NLS}  on $[0,T_{\ast})$ satisfying  $u(0)=u_{0}$ and $u\in C([0, T_{\ast}), H^{1}({\mathbb{R}}))$. Furthermore, for all $t\in [0,T_{\ast})$, the solution satisfies the conservation of energy and mass
\begin{equation*}
E_{{\gamma}}(u(t))=E_{{\gamma}}(u_{0})\quad  \text{and} \quad M(u(t))=M(u_{0}),
\end{equation*}
where
\begin{align*}
E_{{\gamma}}(u)&=\frac{1}{2}\|\partial_{x}u\|^{2}_{L^{2}}-\frac{\gamma}{2}|u(0)|^{2}-\frac{1}{p+1}\|u\|_{L^{p+1}}^{p+1},\quad M(u)=\|u\|^{2}_{L^{2}}.
\end{align*}
By the classical Gagliardo-Nirenberg inequality and conservation laws it is not difficult to show that if $1<p<5$, then  global well-posedness of \eqref{NLS} holds in $H^{1}(\mathbb{R})$.

We say that the solution $u(t)$ to \eqref{NLS} scatters in $H^{1}(\R)$ forward in time, if is defined for any $t\in [0, \infty)$ and  there exists $\psi^{+}\in H^{1}(\R)$
such that
\[
\lim_{t\to\infty}\|u(t)-e^{-itH_{\gamma}}\psi^{+}\|_{H^{1}(\R)}=0.
\]

The scattering theory for equation \eqref{NLS} is related with the ground state $Q$ of the free NLS (\eqref{NLS} with $\gamma=0$), which is the unique positive, symmetric and decreasing solution 
of the following  elliptic equation
\begin{equation}\label{Eqp}
\partial^{2}_{x} Q- Q+Q^{p}=0 \quad\text{in} \quad \mathbb{R}.
\end{equation}
We recall that such $Q$ is exponentially decaying at infinity, and
characterized as the unique minimizer for the Gagliardo-Nirenberg inequality (up to symmetries); see \cite{Weinstein1983} for more details.
The ground state plays a key role in the long-time dynamics of \eqref{NLS}.

The behavior of solutions below the ground state level are now well understood for equation \eqref{NLS}. Indeed,
in \cite{MT2}, M. Ikeda and T. Inui found a necessary and sufficient condition on the data below
the ground state to determine the global behavior (i.e., scattering/blow-up) of the solution.
The dichotomy in behaviour of solutions below the ground state is dependent upon the sign of the functional
\begin{align}
P_{\gamma}(u)&=\|\partial_{x}u\|^{2}_{L^{2}}-\frac{\gamma}{2}|u(0)|^{2}-\frac{(p-1)}{2(p+1)} \|u\|^{p+1}_{L^{p+1}}.
\end{align}
The scattering result of M. Ikeda and T. Inui \cite{MT2} is contained in the next theorem.

\begin{theorem}[Sub-threshold scattering, \cite{MT2}]\label{Th1}
Let $\gamma<0$.  Let $u(t)$ be the corresponding solution to \eqref{NLS} with initial data $u_{0}\in H^{1}(\R)$. If $u_{0}$ obeys
\begin{equation}\label{SubT}
E_{\gamma}(u_{0})[M(u_{0})]^{\sigma}<E_{0}(Q)[M(Q)]^{\sigma}
\quad\text{and}\quad
P_{\gamma}(u_{0})\geq 0,
\end{equation}
where $\sigma=(p+3)/(p-5)$, then the solution $u(t)$ exists globally and scatters in $H^{1}(\R)$.
\end{theorem}

The theorem above is a consequence of the fact that the solutions to \eqref{NLS} obeys the global spacetime bound
\begin{equation}\label{GlobalBoundB}
\|u\|_{L^{a}_{t}L^{r}_{x}(\R\times\R)}<C(E_{\gamma}(u_{0}), M(u_{0}), E_{0}(Q), M(Q)),
\end{equation}
where $a=\frac{2(p-1)(p+1)}{p+3}$ and $r=p+1$. Note also that if $u_{0}\neq 0$ satisfies the condition \eqref{SubT}, then by \cite[Proposition 2.18]{MT2} the corresponding solution $u(t)$ to \eqref{NLS}
with initial data $u_{0}$ obeys the uniform bound $P_{\gamma}(u(t))\gtrsim _{Q, u_{0}}1$ for all $t\in \R$.

The purpose of this paper is to study  the long time dynamics for \eqref{NLS}
exactly at the mass-energy threshold, i.e. when $E_{\gamma}(u_{0})[M(u_{0})]^{\sigma}=E_{0}(Q)[M(Q)]^{\sigma}$.
With this in mind, we establish our first result.

\begin{theorem}[Failure of uniform space-time bounds at threshold.]\label{BoundTheorem}
Let $\gamma<0$.  Then there exists a sequence of global solutions $u_{n}$ of \eqref{NLS} such that
\[
E_{\gamma}(u_{n})[M(u_{n})]^{\sigma}\nearrow  E_{0}(Q)[M(Q)]^{\sigma}
\quad \text{and}\quad 
P_{\gamma}(u_{n}(0))\to 0
\]
as $n\to \infty$ with
\[
\lim_{n\to \infty}\|u_{n}\|_{L^{a}_{t}L^{r}_{x}(\R\times\R)}=\infty.
\]
\end{theorem}

Theorem \ref{BoundTheorem} shows that the hypothesis \eqref{SubT} is sharp, i.e. when we approach the 
mass-energy threshold, the constant $C$ in \eqref{GlobalBoundB} diverges.

We now state the main result of this paper.

\begin{theorem}[Threshold scattering]\label{Th2}
Let $\gamma<0$.  Let $u(t)$ be the corresponding solution to \eqref{NLS} with initial data $u_{0}\in H^{1}(\R)$. If $u_{0}$ obeys
\begin{equation}\label{Thres}
E_{\gamma}(u_{0})[M(u_{0})]^{\sigma}=E_{0}(Q)[M(Q)]^{\sigma}
\quad\text{and}\quad P_{\gamma}(u_{0})\geq 0,
\end{equation}
then the solution $u(t)$ exists globally and $u\in L^{a}_{t}L^{r}_{x}(\R\times\R)$.
In particular, the solution $u$ scatters in $H^{1}(\R)$.
\end{theorem}

For the classical cubic NLS in dimension $N=3$, a similar result was originally proven by Duyckaerts-Roudenko \cite{DuyckaertsRou2010}.
In \cite{DuyckaertsRou2010}, all possible behaviors of solutions with initial data at mass and energy ground states threshold are classified.
This result was later extended for any dimension and any power of the nonlinearity for the entire intecritical range in \cite{CamposFarahRou2020}.

We remark that the method developed in \cite{DuyckaertsRou2010,CamposFarahRou2020} cannot be applied directly to \eqref{NLS}. The main difficulty concerning \eqref{NLS} is clearly the presence of the delta potential. In particular, we cannot apply scaling techniques to obtain the compactness for the nonscattering solutions. To overcome this problem, we use a approach based on the work of Miao-Murphy-Zheng\cite{MiaMurphyZheng2021},
which considered the $3d$ focusing cubic NLS with a repulsive potential.
We also mention the work of \cite{DuyLandRoude} who  obtained analogous result to Theorem \ref{Th2} for the cubic NLS in the exterior
of a convex obstacle.

\begin{remark}
By an application of \cite[Theorem 2.2; see also Example 3.6]{Mizutani2020} we have that the solution $u$ to \eqref{NLS} obtained in Theorem \ref{Th2}  scatters to a free solution in $H^{1}(\R)$; namely there exist asymptotic states $\psi_{\pm}\in H^{1}(\R)$ so that
\[
\|u(t)-e^{it\partial^{2}_{x}}\psi_{\pm}\|_{H^{1}(\R)}\to 0 \quad \text{as $t\rightarrow\pm \infty$}.
\]
\end{remark}

The plan of this paper is as follows. We fix notations at the end of Section \ref{sec:intro}. In Section \ref{S:preli} we give some results that are necessary for later sections, including Strichartz estimates, localized Virial identities, variational analysis, long time perturbation, and linear profile decomposition. In Section \ref{S:Compactness} we show that if Theorem \ref{Th2} fails, then we can find  a solution $u(t)$ of \eqref{NLS} that verifies \eqref{Thres} and a parameter $x(t)$ such that $\left\{u(t, \cdot+ x(t)): t\geq 0\right\}$ is precompact in $H^{1}(\R)$ (Proposition \ref{Criticalsolution}). In Section \ref{S:Modulation}  
we discuss modulation (Proposition \ref{Modilation11}), which plays a vital role in the proof of Theorem \ref{Th2}. 
In Section \ref{S:Impossibility}, we study the behavior of the space translation $x(t)$: More specifically, in Subsections \ref{Sub:11} and \ref{Sub:22} we show that 
if it is bounded then it must be unbounded (Proposition \ref{BoundedUN}), and vice versa (Proposition \ref{NaoBounded});  that is, we show the impossibility  of the `compact' solution $u(t)$ 
established in Section \ref{S:Compactness}, which implies Theorem \ref{Th2}. 
Finally, in Section \ref{BoundL} we prove Theorem \ref{BoundTheorem}.

\textbf{Notations.} We begin with a few remarks on our notation. Given two positive quantities $A$, $B$ we write $A\lesssim B$ or $B\gtrsim A$  to signify $A\leq CB$ for some positive constant  $C$. If  $A \lesssim B \lesssim A$, then we write $A\sim B$. 

For any interval $I\subset \R$, we often write $L^{r}(I)$ to denote the Banach space of functions $f: I:\to \C$ such that the norm
\[
\|f\|_{L^{r}(I)}=\( \int_{\R}|f(x)|^{r}dx\)^{\frac{1}{r}},
\]
is finite, with the usual adjustments when $r=\infty$. 
We write
\[
\|\varphi\|_{H}^2 = \|\partial_x \varphi\|_{L^2}^2 -\gamma |\varphi(0)|^2\quad \text{
for $\varphi\in H^{1}(\R)$}.
\]
Given $p \geq 1$, we denote by $p^{\prime}$ its dual exponent.

\section{Preliminaries}\label{S:preli}

To prove our results, in this section we establish some useful results.

\subsection{Linear Estimates and Local theory}
We fix from now on the following Lebesgue exponent:
\[
r:=p+1, \quad a:=\frac{2(p-1)(p+1)}{p+3}
\quad \text{and}\quad 
b:=\frac{2(p-1)(p+1)}{(p-1)^{2}-(p-1)-4}.
\]
For $\gamma<0$, the Schr\"odinger group obeys the following dispersive estimate:
\[
\|e^{-itH_{\gamma}}f\|_{L^{1}_{x}(\R)\to L^{\infty}_{x}(\R)}\lesssim |t|^{-\frac{1}{2}}.
\]
This estimate implies the following linear estimates that will be fundamental in our study (see \cite[Section 3.1]{BaniVisci2016}):
\begin{align*}
	\|e^{-itH_{\gamma}}f\|_{L^{a}_{t}L^{r}_{x}}&\lesssim \|f\|_{H^{1}},\\
	\|e^{-itH_{\gamma}}f\|_{L^{p-1}_{t}L^{\infty}_{x}}&\lesssim \|f\|_{H^{1}},\\
\|\int^{t}_{0}e^{-i(t-s)H_{\gamma}}g(s)ds\|_{L^{a}_{t}L^{r}_{x}}&\lesssim \|g\|_{L^{b^{\prime}}_{t}L^{r^{\prime}}_{x}},\\
\|\int^{t}_{0}e^{-i(t-s)H_{\gamma}}g(s)ds\|_{L^{p-1}_{t}L^{\infty}_{x}}&\lesssim \|g\|_{L^{b^{\prime}}_{t}L^{r^{\prime}}_{x}}.
\end{align*}

In the following result we have a sufficiently condition for scattering (see \cite[Proposition 3.1]{BaniVisci2016}).
\begin{proposition}\label{ScatterCondi}
Let $u_{0}\in H^{1}(\R)$ and $u$ be the corresponding solution of the Cauchy problem \eqref{NLS} with initial data $u(0)=u_{0}$.
Assume that $u$ is global. If $u\in L^{a}_{t}L^{r}_{x}(\R\times\R)$, then $u$ scatters.
\end{proposition}

We will recall the linear profile decomposition, which is important to study the scattering properties of solutions to
\eqref{NLS}; see \cite[Section 2.2]{BaniVisci2016}. 

\begin{proposition}[Linear profile decomposition] 
\label{LPD}
Let $\{ \varphi_n\}_{n\in \N}$ be a bounded sequence in $H^1(\mathbb{R})$. Then, up to subsequence, we can write 
\begin{align*}
	\varphi_n=\sum_{j=1}^{J} e^{i t_{n}^j H_\gamma} \tau_{x_n^j} \psi^j +R_n^J, \quad \forall J \in \N, 
\end{align*}
where 
$t_n^j\in \mathbb{R}$, $x_n^j \in \mathbb{R}$, $\psi^j \in H^1(\mathbb{R})\setminus\{0\}$, and the following hold. 
\begin{itemize}
\item for any fixed $j$, we have :
\begin{align*}
&\text{either } t_n^j=0 \text{ for any } n \in \N, \text{ or } t_n^j \to \pm \infty \text{ as } n\to \infty,
\\
&\text{either } x_n^j=0 \text{ for any } n \in \N, \text{ or } x_n^j \to \pm \infty \text{ as } n\to \infty.
\end{align*}
\item orthogonality of the parameters:
\[ |t_n^j -t_n^k|+|x_n^j-x_n^k| \to \infty \text{ as } n \to \infty, \quad \forall j\neq k. \]
\item smallness of the reminder:
\[ \forall \eps >0, \exists J=J(\eps) \in \N \text{ such that } \limsup_{n\to \infty} \| e^{-itH_\gamma}R_n^J\|_{L_{t,x}^\infty} <\eps. \]
\item orthogonality in norms: for any $ J\in \N$
\begin{align*} 
\| \varphi_n\|_{L^2}^2 &=\sum_{j=1}^J \|\psi^j\|_{L^2}^2 +\| R_n^J\|_{L^2}^2 +o_n(1),
\\
\| \varphi_n\|_{H}^2 &=\sum_{j=1}^J \| \tau_{x_n^j} \psi^j \|_{H}^2 +\| R_n^J\|_{H}^2 +o_n(1).
\end{align*}
Moreover, we have 
\[ \| \varphi_n\|_{L^q}^q =\sum_{j=1}^J \| e^{i t_n^j H_\gamma} \tau_{x_n^j} \psi^j \|_{L^q}^q +\| R_n^J\|_{L^q}^q +o_n(1), \quad q\in (2,\infty),   \quad \forall J\in \N. \]
\end{itemize}
\end{proposition}

\begin{remark}
\label{rmk1}
Since it follows from the H\"{o}lder inequalities that 
\[\|f\|_{L_t^a L_x^r} \leq \|f\|_{L_t^qL_x^r}^{\theta} \| f\|_{L_t^{\infty}L_x^r}^{1-\theta} \leq \|f\|_{L_t^qL_x^r}^{\theta} \| f\|_{L_t^{\infty}L_x^2}^{(1-\theta)\eta} \| f\|_{L_{t,x}^{\infty}}^{(1-\theta)(1-\eta)},\] 
where $q=\frac{4(p+1)}{p-1}$, $\theta =\frac{2(p+3)}{(p-1)^2}$, and $\eta=\frac{2}{p+1}$, the smallness of the remainder also holds for $L_t^a L_x^r$-norm by the Strichartz estimate. 
\end{remark}

We also will need the following perturbation result; see \cite[Proposition 3.2.]{BaniVisci2016} for more details.

\begin{lemma}[Long time perturbation]
\label{perturb}
For any $M>0$, there exist $\varepsilon=\varepsilon(M)>0$ and a positive constant $C=C(M)$ such that the following occurs. 
Let $v: I\times \R\to \C$ be a solution of the integral equation with source term $e$:
\[ v(t)=e^{-itH_{\gamma}}\varphi +i\int_{0}^{t} e^{-i(t-s)H_\gamma} (|v(s)|^{p-1}v(s))ds +e (t)\]
with $\|v\|_{L_t^a L_x^r(I\times \R)}<M$ and $\|e\|_{L_t^{a}L_x^r(I\times \R)}<\varepsilon$. Assume moreover that $u_0 \in H^1(\mathbb{R})$ is such that $\|u_0-\varphi\|_{H^{1}} < \varepsilon$, then the solution $u: I\times \R\to \C$ to \eqref{NLS} with initial data $u_0$:
\[ u(t)=e^{-itH_{\gamma}}u_{0} +i \int_{0}^{t} e^{-i(t-s)H_{\gamma}}(|u(s)|^{p-1}u(s)) ds,\]
satisfies $u \in L_t^{a}L_x^r(I\times \R)$ and moreover $\|u-v\|_{L_t^{a}L_x^r(I\times \R)}<C\varepsilon$. 
\end{lemma}

\subsection{Varational analysis}
It is well-known that the ground state $Q$ satisfies the Pohozaev's identities
\begin{align}\label{PIh}
\int_{\mathbb{R}}|\partial_{x} Q|^{2}dx&=\frac{p-1}{2(p+1)}\int_{\mathbb{R}}Q^{p+1}\,dx,\\\label{PIh2}
\int_{\mathbb{R}}| Q|^{p+1}dx&=\frac{2(p+1)}{p+3}\int_{\mathbb{R}}Q^{2}\,dx.
\end{align}
We have the following sharp  Gagliardo-Nirenberg inequality, 
\begin{equation}\label{GI}
\|f\|^{p+1}_{L^{p+1}}\leq C_{GN}\|\partial_{x} f\|^{\frac{(p-1)}{2}}_{L^{2}}\|f\|^{\frac{(p+3)}{2}}_{L^{2}},
\end{equation}
where 
\begin{equation}
\label{C_GN}
C_{GN}=\left(\frac{p-1}{p+3}\right)^{\frac{4-(p-1)}{4}}\frac{2(p+1)}{(p-1)\|Q\|^{p-1}_{L^{2}}}.
\end{equation}

By using the Pohozaev's identities, it is easy to derive that
\begin{align}\label{ELQ}
E_{0}(Q)=\frac{p-5}{2(p+3)}\| Q\|^{2}_{L^{2}}=\frac{p-5}{2(p-1)}\|\partial_{x} Q\|^{2}_{L^{2}}
=\frac{p-5}{4(p+1)}\|Q\|^{p+1}_{L^{p+1}}.
\end{align}

In particular, $E_{0}(Q)>0$ (recall that $p>5$).  

The scattering threshold is related with the following minimization problem:
\begin{equation}\label{dinf}
d_{\gamma}:=\inf\left\{S_{\gamma}(\varphi): \varphi\in H^{1}(\R)\setminus\left\{{0}\right\}, P_{\gamma}(\varphi)=0 \right\},
\end{equation}
where
\begin{align}\label{FS}
S_{\gamma}(u)
&=\frac{1}{2}\|\partial_{x}u\|^{2}_{L^{2}}+\frac{1}{2}\|u\|^{2}_{L^2}-\frac{\gamma}{2}|u(0)|^{2}-\frac{1}{p+1}\|u\|_{L^{p+1}}^{p+1}.
\end{align}
The functional $S_{\gamma}$ is often called action. Proposition 1.2 of \cite{MT2} gives the following result.
\begin{lemma}\label{minimumS}
Assume $\gamma<0$. Then $d_{\gamma}=S_{0}(Q)$;  but the infimum  \eqref{dinf} is never attained.
\end{lemma}

We will need the following Lemma. We recall that
$\|u\|_{H}^2 = \|\partial_x u\|_{L^2}^2 -\gamma |u(0)|^2$ for $u\in H^{1}(\R)$.

\begin{lemma}
\label{variational}
Let $\varphi \in H^1(\mathbb{R})$. 
If $\|\varphi\|_{L^2}^{\sigma} \|\varphi\|_{H} \leq \|Q\|_{L^2}^{\sigma} \|\partial_x Q\|_{L^2}$,  then we have $P_{\gamma}(\varphi) \geq 0$.
Moreover, if $[M(\varphi)]^{\sigma} E_{\gamma}(\varphi) \leq [M(Q)]^{\sigma}E_{0}(Q)$ and $P_{\gamma}(\varphi) \geq 0$, then we have $\|\varphi\|_{L^2}^{\sigma} \|\varphi\|_{H} \leq \|Q\|_{L^2}^{\sigma} \|\partial_x Q\|_{L^2}$. 
\end{lemma}

\begin{proof}
Assume $\|\varphi\|_{L^2}^{\sigma} \|\varphi\|_{H} \leq \|Q\|_{L^2}^{\sigma} \|\partial_x Q\|_{L^2}$. By the Gagliardo-Nirenberg inequality and the assumption, we have
\begin{align*}
	P_{\gamma} (\varphi)
	&\geq \|\partial_x \varphi\|_{L^2}^2 -\frac{\gamma}{2} |\varphi(0)|^2 -\frac{(p-1)C_{GN}}{2(p+1)} \|\varphi\|_{L^2}^{\frac{p+3}{2}} \|\partial_x \varphi \|_{L^2}^{\frac{p-1}{2}} 
	\\
	&=\|\partial_x \varphi\|_{L^2}^2 \left(1 - \frac{(p-1)C_{GN}}{2(p+1)} \|\varphi\|_{L^2}^{\frac{p+3}{2}} \|\partial_x \varphi \|_{L^2}^{\frac{p-5}{2}}  \right) -\frac{\gamma}{2} |\varphi(0)|^2
	\\
	&\geq \|\partial_x \varphi\|_{L^2}^2 \left(1 - \frac{(p-1)C_{GN}}{2(p+1)} \|Q\|_{L^2}^{\frac{p+3}{2}} \|\partial_x Q\|_{L^2}^{\frac{p-5}{2}}  \right) -\frac{\gamma}{2} |\varphi(0)|^2
\end{align*}
By \eqref{PIh}, \eqref{PIh2}, and \eqref{C_GN}, we get
\begin{align*}
	\frac{(p-1)C_{GN}}{2(p+1)} \|Q\|_{L^2}^{\frac{p+3}{2}} \|\partial_x Q \|_{L^2}^{\frac{p-5}{2}}  =1.
\end{align*}
Therefore, we have $P_{\gamma} (\varphi) \geq  -\frac{\gamma}{2} |\varphi(0)|^2 \geq 0$. 

Next, assume $P_{\gamma}(\varphi) \geq 0$. By \eqref{ELQ} and $M(Q)^{\sigma} E_{0}(Q) \geq M(\varphi)^{\sigma}E_{\gamma}(\varphi)$, we obtain
\begin{align}
\label{eq1}
	\frac{p-5}{2(p-1)} M(Q)^{\sigma} \|\partial_x Q\|_{L^2}^2 
	= M(Q)^{\sigma} E_{0}(Q) \geq M(\varphi)^{\sigma}E_{\gamma}(\varphi). 
\end{align}
By the assumption, we have
\begin{align}
\label{eq1.2}
	\notag
	E_{\gamma}(\varphi) \geq 
	E_{\gamma}(\varphi) - \frac{2}{p-1}P_{\gamma}(\varphi)
	&=\frac{p-5}{2(p-1)} (\|\partial_x \varphi\|_{L^2}^2-\gamma|\varphi(0)|^2) -\frac{\gamma}{p-1}|\varphi(0)|^2
	\\ 
	&\geq \frac{p-5}{2(p-1)}\|\varphi\|_{H}^2.
\end{align}
Therefore, combining this with \eqref{eq1}, we get
\begin{align*}
	\frac{p-5}{2(p-1)} M(Q)^{\sigma} \|\partial_x Q\|_{L^2}^2  
	\geq  \frac{p-5}{2(p-1)} M(\varphi)^{\sigma} \|\varphi\|_{H}^2.
\end{align*}
This completes the proof. 
\end{proof}

\begin{lemma}\label{GlobalS}
Assume $\gamma<0$. If $u_{0}\in H^{1}(\R)$ satisfies
\begin{equation}\label{Condition11}
E_{\gamma}(u_{0})=E_{0}(Q), \quad M(u_{0})=M(Q)
\quad \text{and}\quad 
P_{\gamma}(u_{0})\geq 0,
\end{equation}
then the corresponding solution $u(t)$ to \eqref{NLS} is global in time and satisfies 
\begin{equation}\label{PositiveP}
P_{\gamma}(u(t))> 0 \quad \text{for all $t\in \R$}.
\end{equation}
Moreover,  we also have 
\begin{equation}\label{Gradient}
\|u(t)\|^{2}_{H}<\|\partial_{x}Q\|^{2}_{L_{x}^{2}}
\quad \text{for all $t\in \R$}.
\end{equation}
In particular, we have
\begin{equation}\label{EquiNorm}
\sup_{t\in \R}[\|u(t)\|^{2}_{H}+\|u\|^{2}_{L^{2}_{x}}]
\sim
 S_{0}(Q).
\end{equation}

\end{lemma}
\begin{proof}
First, suppose by contradiction that there exists $t_{0}>0$ such that $P_{\gamma}(u(t_{0}))=0$. From \eqref{Condition11}
we have $S_{\gamma}(u(t_{0}))=S_{0}(Q)$, where $S_{\gamma}$ is given by \eqref{FS}. Thus, $u(t_{0})$ is a minimizer of the variational problem \eqref{dinf}, 
which is a contradiction with Lemma \ref{minimumS}. Therefore 
\begin{equation}\label{PPositive}
\text{$P_{\gamma}(u(t))> 0$ for all $t$ in the existence time.}
\end{equation}
The proof of Lemma \ref{variational} and $P_{\gamma}(u(t))>0$ imply $\|u(t)\|_{H}^2 < \|Q\|_{L^2}^2$. From them, we also have 
\[
S_{0}(Q)=S_{\gamma}(u(t))
\gtrsim\|\partial_{x}u(t)\|^{2}_{L_{x}^{2}}+\|u(t)\|^{2}_{L^{2}_{x}}-\gamma|u(t, 0)|^{2}.
\] 
The inverse is trivial. We completes the proof.
\end{proof}

\subsection{Localized Virial identity}
Given $R>1$, we define
\[
w_{R}(x)=R^{2}\phi\(\frac{x}{R}\)
\quad \text{and}\quad w_{\infty}(x)=x^{2},
\]
where $\phi$ is a real-valued and radial function such that
\[
\phi(x)=
\begin{cases}
x^{2},& \quad |x|\leq 1\\
0,& \quad |x|\geq 2,
\end{cases}
\quad \text{with}\quad 
|\partial_{x}^{\alpha}\phi(x)|\lesssim |x|^{2-\alpha}.
\]
Moreover, we define the functional
\[
V_{R}[u]=2\IM\int_{\R}\partial_{x}[w_{R}(x)]\overline{u(t,x)}\partial_{x}u(t,x)dx
\]

In \cite[Section 3]{MT2} (see also \cite[Lemma 4.1]{BaniVisci2016}), the authors proved the following.

\begin{lemma}\label{VirialIden}
Let $R\in [1, \infty]$. Assume that $u(t)$ solves \eqref{NLS}. Then we have
\begin{equation}\label{LocalVirial}
\frac{d}{d t}V_{R}[u]=I_{R, \gamma}[u(t)],
\end{equation}
where
\begin{align*}
I_{R, \gamma}[u]&:=4\int_{\R}\partial^{2}_{x}[w_{R}(x)]|\partial_{x}u(t,x)|^{2}-
2\frac{(p-1)}{p+1}\partial^{2}_{x}[w_{R}(x)]|u(t,x)|^{p+1}dx\\
&-\int_{\R}\partial^{4}_{x}[w_{R}(x)]|u(t,x)|^{2}dx
-2\gamma\left.\partial^{2}_{x}[w_{R}(x)]\right|_{x=0}|u(t,0)|^{2}\\
&-2\gamma\RE\left\{  \left.\partial_{x}[w_{R}(x)]\right|_{x=0}u(t,0)\overline{\partial_{x}u(t,0)} \right\}\\
&=I_{R,0}[u]-4\gamma|u(t,0)|^{2}.
\end{align*}
In particular, if $R=\infty$ we have $I_{\infty, \gamma}[u]=8P_{\gamma}(u)$.
\end{lemma}

The proof of the following lemma is the same as the one given in \cite[Lemma 2.9]{MiaMurphyZheng2021} and will be omitted.
\begin{lemma}\label{Virialzero}
Let $R\in [1, \infty]$, $\theta\in \R$ and $y\in\R$. Then
\[
I_{R,0}[e^{i\theta}Q(\cdot-y)]=0.
\]
\end{lemma}

As a direct consequence of the Lemma \ref{Virialzero}, we get the following result (see \cite[Corollary 2.10]{MiaMurphyZheng2021}).

\begin{lemma}\label{VirialModulate}
Let $u$ be the solution of \eqref{NLS} on an interval $I$. Let $R\in [1, \infty]$, $\chi: I\to \R$, $\theta: I\to \R$,
$y: I\to \R$. Then for all $t\in\R$,
\begin{align}\nonumber
	\frac{d}{d t}V_{R}[u]&=I_{\infty,0}[u(t)]\\ \label{Modu11}
	                     &+I_{R, \gamma}[u(t)]-I_{\infty,0}[u(t)]\\\label{Modu22}
											 &-\chi(t)\big\{I_{R,0}[e^{i\theta(t)}Q(\cdot-y(t))]-I_{\infty,0}[e^{i\theta(t)}Q(\cdot-y(t))]\big\}.
\end{align}
\end{lemma}

\section{Compactness properties}\label{S:Compactness}

In the following result we show that if Theorem \ref{Th2} fails, then there exists a 
solution of \eqref{NLS} at the threshold \eqref{Thres} with infinite $L_{t}^{a}L^{r}_{x}([0, \infty)\times\R)$-norm, which is precompact in $H^{1}$ modulo some time-dependent spatial center.

\begin{proposition}\label{Criticalsolution}
Suppose Theorem \ref{Th2} fails for some $\gamma<0$. Then we find a forward global 
solution $u\in C([0, \infty); H^{1}(\R))$ of \eqref{NLS} with initial data $u_{0}$ which satisfies
\begin{align}\label{Su}
	&E_{\gamma}(u_{0})=E_{0}(Q), \quad M(u_{0})=M(Q)
\quad \text{and}\quad 
P_{\gamma}(u_{0})\geq 0,\\ \label{Blowup}
	&	\|u \|_{L_{t}^{a}L^{r}_{x}([0, \infty)\times\R)}=\infty.
\end{align}
Moreover, there exists a  function $x_{0}:[0, \infty) \to \R$ such that
$\left\{u(t, \cdot+x_{0}(t)):t\in [0, \infty)\right\}$ is pre-compact in $H^{1}(\R)$.
\end{proposition}

Before showing the proposition above, we first show the following lemma.

\begin{lemma}\label{Newconditions}
Suppose that Theorem \ref{Th2} holds for any $\gamma<0$ with the hypothesis \eqref{Thres} replaced by \eqref{Su}.
Then we can prove the same conclusion in Theorem \ref{Th2} (for any $\gamma<0$) with the original condition \eqref{Thres}.
\end{lemma}
\begin{proof}
Let $\gamma<0$. Suppose that Theorem \ref{Th2} is true with the condition \eqref{Su}. Consider $u_{0}\in H^{1}(\R)$ such that 
\[
E_{\gamma}(u_{0})[M(u_{0})]^{\sigma}=E_{0}(Q)[M(Q)]^{\sigma}
\quad\text{and}\quad P_{\gamma}(u_{0})\geq 0.
\]
We set $\gamma^{\ast}=\lambda\gamma$, $v_{0}(x)=\lambda^{\frac{2}{(p-1)}} u_{0}(\lambda x)$ and $v(t,x)=\lambda^{\frac{2}{(p-1)}} u(\lambda^{2}t, \lambda x)$, where 
$\lambda^{\frac{4}{(p-1)}-1}=\frac{M(Q)}{M(u_{0})}$. Notice that
\[
E_{\gamma^{\ast}}(v_{0})=\lambda^{\frac{p+3}{p-1}}E_{\gamma}(u_{0})\quad \text{and}\quad 
P_{\gamma^{\ast}}(v_{0})=\lambda^{\frac{p+3}{p-1}}P_{\gamma}(u_{0}).
\]
Thus, since $\lambda^{\frac{p+3}{p-1}}=\(\frac{M(u_{0})}{M(Q)}\)^{\sigma}$, we obtain
\[
E_{\gamma^{\ast}}(v_{0})=E_{0}(Q), \quad M(v_{0})=M(Q)
\quad \text{and}\quad 
P_{\gamma^{\ast}}(v_{0})\geq 0.
\]
Now, we define the rescaled potential $V_{\lambda}(x):=\gamma\lambda^{2}\delta(\lambda x)=\gamma^{\ast}\delta(x)$ with $\gamma^{\ast}=\lambda\gamma$. Then the function $v$ satisfies
\[
i\partial_{t}v+\partial^{2}_{x}v+ \gamma^{\ast}\delta(x)v+|v|^{p-1}v=0.
\]
Since $\gamma^{\ast}=\lambda\gamma<0$, by  hypothesis we get $v\in L_{t}^{a}L^{r}_{x}(\R\times \R)$, which implies that
$u\in L_{t}^{a}L^{r}_{x}(\R\times \R)$. In particular, by  Proposition \ref{ScatterCondi} we obtain that $u$ scatters in $H^{1}(\R)$.
\end{proof}

\begin{proof}[{Proof of Proposition \ref{Criticalsolution}}]
Assume that Theorem \ref{Th2} fails,  by using Lemma \ref{Newconditions} we infer that there exists a $u_{0}\in H^{1}(\R)$ with
\[
E_{\gamma}(u_{0})=E_{0}(Q), \quad M(u_{0})=M(Q)
\quad \text{and}\quad 
P_{\gamma}(u_{0})\geq 0,
\] 
such that if $u$ is the corresponding forward-global solution of \eqref{NLS} with initial data $u_{0}$, then
\[
\|u \|_{L_{t}^{a}L^{r}_{x}([0, \infty)\times\R)}=\infty.
\]
Moreover, by Lemma \ref{GlobalS} we have that $u$ is bounded in $H^{1}(\R)$ and satisfies properties \eqref{PositiveP}-\eqref{Gradient}.
Our goal now is to show that exists a  parameter $x_{0}:[0, \infty) \to \R$ such that
$\left\{u(t, \cdot+x_{0}(t)):t\in [0, \infty)\right\}$ is precompact in $H^{1}(\R)$.  The argument is similar to that given in
\cite[Lemma 3.11]{MT2}.

With this in mind, let $\{\tau_n\}$ be an arbitrary sequence such that $\tau_n \to \infty$. It is enough to show that there exists a sequence $\{x_n\}$ such that $u(\tau_n, x+x_n)$ converges strongly in $H^1(\mathbb{R})$. The sequence $\{u(\tau_n)\}$ is bounded in $H^1(\mathbb{R})$ by \eqref{EquiNorm}. By the linear profile decomposition (Lemma \ref{LPD}), we have, up to subsequence,
\begin{align*}
	u(\tau_n) = \sum_{j=1}^{J} e^{i t_{n}^j H_\gamma} \tau_{x_n^j} \psi^j +R_n^J
\end{align*}
and the properties in the statement hold. We set $\psi_n^j := e^{i t_{n}^j H_\gamma} \tau_{x_n^j} \psi^j$. 

We prove that $J=1$. It is easy to show that $J=0$ does not occur. Indeed, by the linear profile decomposition, if $J=0$, then $\|e^{-it H_{\gamma}} u(\tau_n)\|_{L_{t}^{q}L_x^r(\R\times\R)} \to 0$ as $n \to \infty$ (see Remark \ref{rmk1}). Then, by the long time perturbation (Lemma \ref{perturb}), we get $\|u\|_{L_t^qL_x^r([\tau_n,\infty)\times\R)} \lesssim 1$ for large $n \in \mathbb{N}$. This contradicts the definition of $u$. 

First, by the linear profile decomposition, we have
\begin{align*}
	&\lim_{n \to \infty} \left(\sum_{j=1}^{J} M(\psi_n^j) + M(R_n^J)\right) = \lim_{n \to \infty} M(u(\tau_n)) =M(u_0) =M(Q),
	\\
	&\lim_{n \to \infty} \left(\sum_{j=1}^{J} \| \psi_n^j\|_{H}^2 + \| R_n^J\|_{H}^2\right) = \lim_{n \to \infty} \|u(\tau_n)\|_{H}^2 \leq \|\partial_x Q\|_{L^2}^2.
\end{align*}
Thus, it holds that
\begin{align*}
	\|\psi_n^j\|_{L^2}^{\sigma} \|\psi_n^j\|_{H} \leq \| Q\|_{L^2}^{\sigma}  \|\partial_x Q\|_{L^2}
	\text{ and }
	\|R_n^j\|_{L^2}^{\sigma} \|R_n^j\|_{H} \leq \| Q\|_{L^2}^{\sigma}  \|\partial_x Q\|_{L^2}
\end{align*}
for any $j$ and for large $n$.  From this and the proof of Lemma \ref{variational}, it follows that $E_{\gamma}(\psi_n^j),E_{\gamma}(R_n^J)\geq 0$.
Now, since we also have 
\begin{align*}
	\lim_{n \to \infty} \left(\sum_{j=1}^{J} E_{\gamma}(\psi_n^j) + E_{\gamma}(R_n^J)\right) = \lim_{n \to \infty} E_{\gamma}(u(\tau_n)) =E_{\gamma}(u_0) =E_{0}(Q),
\end{align*}
This and $E_{\gamma}(\psi_n^j),E_{\gamma}(R_n^J)\geq 0$ imply that $E_{\gamma}(\psi_n^j), E_{\gamma}(R_n^J) \leq E_{0}(Q)$.

Assume that $J \geq 2$. Then there exists $\delta>0$ such that $M(\psi_n^j)^{\sigma}E_{\gamma}(\psi_n^j) <M(Q)^{\sigma}E_0(Q) -\delta$. 
By reordering, we may choose $J_1, \cdots, J_4$ such that 
\begin{align*}
	&1 \leq j \leq J_1 \Rightarrow t_n^j =0\ (\forall n \in \mathbb{N}) \text{ and } x_n^j =0 \ (\forall n \in \mathbb{N}),
	\\
	&J_1+1 \leq j \leq J_2 \Rightarrow t_n^j =0\ (\forall n \in \mathbb{N}) \text{ and } |x_n^j|\to \infty \ (n \to \infty),
	\\
	&J_2+1 \leq j \leq J_3 \Rightarrow |t_n^j|\to \infty \ (n \to \infty) \text{ and } x_n^j =0 \ (\forall n \in \mathbb{N}),
	\\
	&J_3+1 \leq j \leq J_4 \Rightarrow |t_n^j|\to \infty \ (n \to \infty) \text{ and } |x_n^j|\to \infty \ (n \to \infty),
\end{align*}
where we are assuming that there is no $j$ such that $a \leq j \leq b$ if $a>b$. We will define nonlinear profiles associated with $\psi_n^j$. If there is no $j$ such that $J_k+1\leq j \leq J_{k+1}$ for some $k\in \{0,1,2,3\}$, where $J_0=0$, then skip the construction of nonlinear profiles in the following steps. 

We first consider the case of $1 \leq j \leq J_1$. By the orthogonality of the parameter $t_n^j$ and $x_n^j$, we note that $J_1=0$ or $1$. (As stated before, skip this step if $J_1=0$.) We define a solution $N$ to \eqref{NLS} with the initial data $N(0)=\psi^1$. Then, the solution $N$ is global and satisfies $\|N\|_{L_t^aL_x^r(\mathbb{R}\times \mathbb{R})} \lesssim 1$ since $M(\psi^1)^{\sigma} E_{\gamma}(\psi^1) <M(Q)^{\sigma} E_0(Q) -\delta$ and $P_{\gamma}(\psi^1)\geq 0$ hold and imply the scattering result (see \cite[Lemma 3.11]{MT2}). 

We consider the case of $J_1+1 \leq J_2$. We define a solution $U^j$ to the usual nonlinear Schr\"{o}dinger equation (i.e. the equation with $\gamma=0$) with the initial data $\psi^j$. 
It holds that $\|U^j\|_{L_t^aL_x^r(\mathbb{R}\times \mathbb{R})} \lesssim 1$ (see \cite{MT2}). We set $U_n^j(t,x):=U^j(t,x-x_n^j)$. 

We consider the case of $J_2+1 \leq j \leq J_3$. If $j$ satisfies $t_n^j \to -\infty$, then we define a solution $W^j$ to \eqref{NLS} that scatters to $\psi^j$ as $t \to + \infty$. If $j$ satisfies $t_n^j \to +\infty$, then we define a solution $W^j$ to \eqref{NLS} that scatters to $\psi^j$ as $t \to - \infty$. 
Then, in each cases,  $W^j$ is global in both time directions and $\|W^j \|_{L_t^aL_x^r(\mathbb{R}\times \mathbb{R})} \lesssim 1$ (see \cite{MT2}). We set $W_{n}^j (t,x):=W^j (t-t_n^j,x)$. 

We consider the case of $J_3 +1 \leq j \leq J_4$. If $j$ satisfies $t_n^j \to -\infty$, then we define a solution $V^j$ to the usual NLS that scatters to $\psi^j$ as $t \to + \infty$. If $j$ satisfies $t_n^j \to +\infty$, then we define a solution $V^j$ to the usual NLS that scatters to $\psi^j$ as $t \to - \infty$. Then, in each cases,  $V^j$ is global in both time directions and $\|V^j \|_{L_t^aL_x^r(\mathbb{R}\times \mathbb{R})} \lesssim 1$ (see \cite{MT2}). We set $V_{n}^j (t,x):=V^j (t-t_n^j,x-x_n^j)$. 
We denote all functions $N,U_n^j, W_n^j,V_n^j$ by $v_n^j$. We define nonlinear profile as $Z_n^J:=\sum_{j=1}^{J} v_n^j$.
In the same way as in \cite{MT2}, we have
\begin{align*}
	Z_n^J=e^{-itH_{\gamma}}\varphi_n +i\int_{0}^{t}  e^{-i(t-s)H_{\gamma}} |Z_n^J|^{p-1}Z_n^J ds -e^{-itH_{\gamma}} R_n^J+ s_n^J
\end{align*}
with $\|s_n^J\|_{L_t^aL_x^r} \to 0$ as $n \to \infty$ and $\limsup_{n \to \infty}\|e^{-itH_{\gamma}} R_n^J\|_{L_t^aL_x^r} < \varepsilon$ for large $J$. Moreover, $\limsup_{n \to \infty} \|Z_n^J\|_{L_t^aL_x^r}$ is bounded independently on $J$. by the long time perturbation, we obtain $\|u(\tau_n)\|_{L_t^aL_x^r} \lesssim  1$. This is a contradiction. Therefore, we get $J=1$. 

Therefore, we get
\begin{align*}
	u(\tau_n) =e^{it_n H_{\gamma}} \tau_{x_n} \psi +R_n
\end{align*}
and $\lim_{n \to \infty}\|R_n\|_{H^1} =0$. Assuming $|t_n| \to \infty$, we derive a contradiction to the non-scattering of $u$ by the standard argument. Therefore,  $u(\tau_n, x+x_n)=\psi(x) +R_n (x+x_n)$ and thus $u(\tau_n, \cdot +x_n)$ strongly converges to $\psi$ in $H^1(\mathbb{R})$. The standard argument (see, e.g., \cite[Subsection 3.2]{MiaMurphyZheng2021}) implies the existence of $x_0$ satisfying that $\{u(t,\cdot+x_{0}(t)): t \in [0,\infty)\}$ is pre-compact in $H^1(\mathbb{R})$. 
\end{proof}

\section{Modulation analysis}\label{S:Modulation}

Let $u$ be the forward solution given in Proposition \ref{Criticalsolution}. We recall that
\begin{equation}\label{22condition}
E_{\gamma}(u_{0})=E_{0}(Q), \quad M(u_{0})=M(Q)
\quad \text{and}\quad 
P_{\gamma}(u_{0})\geq 0.
\end{equation}
Notice that by Lemma \ref{GlobalS} we have
\begin{equation}\label{deltapositive}
\mu(t):=\mu(u(t))>0, \quad \text{where} \quad 
\mu(v)=\|\partial_{x}Q\|^{2}_{L_{x}^{2}}-[\|\partial_{x}v\|^{2}_{L_{x}^{2}}-\gamma|v(0)|^{2}].
\end{equation}

For $\mu_{0}>0$ small, we define the set
\[
I_{0}=\left\{t\in [0, \infty):\mu(u(t))<\mu_{0}\right\}.
\] 

The following proposition is the main result of this section.
The goal of this result is to construct modulation parameters $\theta(t)$ and $y(t)$ such that the quantity
$\mu(t)$ controls  $\| u(t)-e^{i\theta(t)}Q(\cdot-y(t))\|_{H^{1}}$ as well as the  potential $-\gamma|u(t,0)|^{2}$, the 
parameter $y(t)$ and its derivative $y^{\prime}(t)$.

\begin{proposition}\label{Modilation11}
Let $\gamma$ be strictly negative. Then there exist $\mu_{0}>0$ sufficiently small and two functions $\theta: I_{0}\to \R$ and $y: I_{0}\to \R$ such that we can write
\begin{equation}\label{DecomU}
u(t,x)=e^{i \theta(t)}[g(t)+Q(x-y(t))]\quad \text{for all $t\in I_{0}$},
\end{equation}
and the following holds:
\begin{equation}\label{Estimatemodu}
\frac{e^{-2|y(t)|}}{|y(t)|^{2}}+|y^{\prime}(t)|+
[-\gamma|u(t,0)|^{2}]^\frac{1}{2}
\lesssim\mu(t)\sim \|g(t)\|_{H^{1}} \quad \text{for all $t\in I_{0}$}.
\end{equation}
\end{proposition}

The proof of Proposition above borrows ingredients from \cite{DuyLandRoude, MiaMurphyZheng2021}.
Before proving Proposition \ref{Modilation11}, some preliminaries are needed.

\begin{remark}
The modulation bound $|y(t)|^{-2} e^{-2|y(t)|} \lesssim\mu(t)$ in Proposition \ref{Modilation11} will not be necessary in our analysis.
However, this bound is interesting in its own right and potentially useful in future work.
\end{remark}

\begin{lemma}\label{LemmaMod}
For any $\epsilon>0$ small, there exists $\mu_{0}=\mu_{0}(\epsilon)$ sufficiently small such that if
$\mu(u(t))<\mu_{0}$, then there exists $(\theta_{0}(t), y_{0}(t))\in \R^{2}$ such that
\begin{equation}\label{CondiModu}
\| u(t)-e^{i\theta_{0}(t)}Q(\cdot-y_{0}(t))\|_{H^{1}}<\epsilon.
\end{equation}
\end{lemma}
\begin{proof}
The result is proved by contradiction. Assume that there exist $\epsilon>0$ and a sequence 
$\left\{t_{n}\right\}\subset \R$ such that
\begin{equation}\label{Contra11}
\mu(u(t_{n}))\to 0, \quad
\inf_{\theta\in \R}\inf_{y\in \R}\|u(t_{n})-e^{i\theta}Q(\cdot-y)\|_{H^{1}}\geq \epsilon.
\end{equation}
Since $\delta(u(t_{n}))\to 0$, from \eqref{22condition} we obtain  
\[
S_{\gamma}(u(t_{n}))  \to S_{0}(Q)\quad \text{and}\quad  N_{\gamma}(u(t_{n}))\to N_{0}(Q)=0,
\]
where $N_{\gamma}$ is the Nehari functional given by
\[
N_{\gamma}(v)=\|\partial_{x}v\|^{2}_{L^{2}}+\|v\|^{2}_{L^{2}}-{\gamma}|v(0)|^{2}-\|v\|_{L^{p+1}}^{p+1}
\quad \text{for $v\in H^{1}(\R)$}.
\]
This means that $N_{0}(u(t_{n})) \leq 0$ for large $n$.
That is, $\left\{u(t_{n})\right\}$ is a minimizing sequence of problem 
\[
S_0(Q)=\inf\left\{S_{0}(f): f\in H^{1}(\R),  N_{0}(f) \leq 0\right\}.
\]
Therefore, there exist $(\theta_{n}, y_{n})\in \R^{2}$ such that
$e^{i\theta_{n}}u(t_{n},\cdot+y_{n})\to Q$ in $H^{1}(\R)$, which is a contradiction with \eqref{Contra11} 
and finishes the proof.
\end{proof}

\begin{remark}\label{Yinfinity}
Let $R\geq1$. Notice that if $\mu_{0}$ is sufficiently small in Lemma \ref{LemmaMod} we can guarantee that
\begin{equation}\label{Nobound}
|y_{0}(t)|\geq R\quad \text{for $t\in \R$}.
\end{equation}
Indeed, suppose that \eqref{Nobound} is false. Then there exists a sequence of times $\left\{t_{n}\right\}$
such that
\begin{equation}\label{Bounddelta}
\text{ $\mu(t_{n})\to0$ and $|y_{0}(t_{n})|\leq R$ for all $n\in \N$}.
\end{equation}
As $\mu(t_{n})\to0$, by using  \eqref{CondiModu} 
we see that (recall that \eqref{EquiNorm})
\[
-\gamma\lim_{n\to \infty}|u(t_{n}, 0)|^{2}=\lim_{n\to \infty}[\|\partial_{x}Q\|^{2}_{L^{2}}-\|\partial_{x}u(t_{n})\|^{2}_{L^{2}}]
=0.
\]
Again by \eqref{CondiModu},  
passing to a subsequence, we have
\[
|Q(y_{0}(t_{n}))|^{2}\to 0 \quad \text{as $n\to\infty$},
\]
which is a contradiction because the sequence $\left\{y_{0}(t_{n})\right\}$ is bounded.
\end{remark}

By an application of implicit function theorem and Lemma \ref{LemmaMod} we have the following result.

\begin{lemma}\label{ExistenceF}
There exist $\mu_{0}>0$ and two functions $\theta: I_{0}\to \R$ and $y: I_{0}\to \R$ such that if
$\mu(t)=\mu(u(t))<\mu_{0}$, then
\begin{equation}\label{Taylor}
\| u(t)-e^{i\theta(t)}Q(\cdot-y(t))\|_{H^{1}} \ll 1.
\end{equation}
Furthermore, the function $g(t):=g_{1}(t)+i g_{2}(t)=e^{-i\theta(t)}[u(t)-e^{i\theta(t)}Q(\cdot-y(t))]$ satisfies
\begin{equation}\label{Ortogonality}
\< g_{2}(t), Q(\cdot-y(t)) \>=\< g_{1}(t), \partial_{x}Q(\cdot-y(t))\>\equiv 0.
\end{equation}
\end{lemma}
\begin{proof}
The proof closely follows the proof of \cite[Lemma 5.3]{MiaMurphyZheng2021}
\end{proof}

By \eqref{Taylor}, a Taylor expansion gives 
\[
\begin{split}
S_{0}(u(t))-S_{0}(Q)=S_{0}(e^{-i\theta(t)}u(t))-S_{0}(Q(\cdot-y(t)))\\
=\< S^{\prime}_{0}(Q(\cdot-y(t)))[g(t)], g(t) \>
+\frac{1}{2}\< S^{\prime\prime}_{0}(Q(\cdot-y(t)))[g(t)], g(t)\>+
o(\|g\|^{2}_{H^{1}}).
\end{split}
\]
Since $S^{\prime}_{0}(Q(\cdot-y(t))=0$, we obtain
\begin{equation}\label{Tay22}
S_{0}(u(t))-S_{0}(Q)=\frac{1}{2}\< S^{\prime\prime}_{0}(Q)[g(t,\cdot+y(t) )], g(t, \cdot+y(t))\>+
o(\|g\|^{2}_{H^{1}})
\end{equation}

Notice that the operator $S^{\prime\prime}_{0}(Q)$ can be separated into a real and an imaginary part
$L_{1}$ and $L_{2}$ such that
\[
\< S^{\prime\prime}_{0}(Q)w, w)\>=\<L_{1}u,u\>+\<L_{2}v,v \>.
\]
where $w=u+iv\in H^{1}(\R)$. Here,  $L_{1}$ and $L_{2}$  are two bounded operators defined in $H^{1}(\R)$ given by
\begin{align*}
L_{1}u&=-\partial^{2}_{x}u+u-pQ^{p-1}u,\\
L_{2}u&=-\partial^{2}_{x}v+v-Q^{p-1}v.	
\end{align*}

Lemma 3.5 of \cite{CamposFarahRou2020} gives the following result.

\begin{lemma}\label{CoerS}
There exists $C>0$ such that for every $h=h_{1}+ih_{2}\in H^{1}(\R)$ satisfying
\begin{equation}\label{Qorto}
\<h_{1}, \partial_{x}Q\>=\<h_{1}, Q^{p}\>=\<h_{2}, Q\>=0
\end{equation}
we have
\[
\< S^{\prime\prime}_{0}(Q)[h], h\>\geq C \|h\|^{2}_{H^{1}}.
\]
\end{lemma}

Notice that the orthogonality condition is different from those in \cite{DuyLandRoude, MiaMurphyZheng2021}. In those, they assumes $\<h_{1}, \Delta Q\>=0$ instead of $\<h_{1}, Q^{p}\>=0$. As pointed in \cite{CamposFarahRou2020}, we use the orthogonality condition $\<h_{1}, Q^{p}\>=0$ in the one dimensional case.

\begin{lemma}\label{BoundI}Fix $\gamma<0$. Let $(\theta(t), y(t))$ and $g(t)$ be as in Lemma \ref{ExistenceF}. Then we have
\begin{equation}\label{DeltaBound}
-\gamma|u(t,0)|^{2}\lesssim \mu^{2}(t)\sim \|g(t)\|^{2}_{H^{1}}.
\end{equation}
\end{lemma}
\begin{proof}
Since $S_{\gamma}(u(t))=S_{0}(Q)$, from \eqref{Tay22} we get
\begin{equation}\label{Aprox}
0=\frac{1}{2}\< S^{\prime\prime}_{0}(Q)[g(\cdot+y)], g(\cdot+y)\>-\frac{\gamma}{2}|u(t,0)|^{2}+
o(\|g\|^{2}_{H^{1}}).
\end{equation}
We decompose $g(t)$ as follows
\begin{equation}\label{Defg}
g=\lambda Q(\cdot-y)+h, \quad \text{where} \quad \lambda=\frac{\<g_{1}(\cdot+y), Q^{p}\>}{\<Q, Q^{p}\>}.
\end{equation}
It is clear that $\lambda\in \R$. Moreover, from \eqref{Taylor} and definition of $\lambda$ given in \eqref{Defg} we get
\begin{equation}\label{BoundR}
|\lambda|\lesssim \|g\|_{H^{1}}\ll 1.
\end{equation}
Now, by \eqref{Ortogonality} we deduce
\begin{equation}\label{Orto22}
\< h_{2}(\cdot+y), Q \>=\< h_{1}(\cdot+y), \partial_{x}Q\>=\< h_{1}(\cdot+y), Q^{p}\>=0.
\end{equation}
Then, Lemma \ref{CoerS} implies
\[
\< S^{\prime\prime}_{0}(Q)[h(\cdot+y)], h(\cdot+y)\> \gtrsim \|h\|^{2}_{H^{1}}.
\]
Combining estimate above and \eqref{Aprox} we obtain
\[
\|h\|^{2}_{H^{1}}-\gamma|u(t,0)|^{2}\lesssim \lambda^{2}+|\lambda\<L_{1}Q, h_{1}(\cdot+y) \>|+
o(\|g\|^{2}_{H^{1}}).
\]
Notice that \eqref{Orto22} and \eqref{Eqp} implies $\<-\partial^{2}_{x}Q, h_{1}(\cdot+y)\>=-\<Q, h_{1}(\cdot+y)\>$. Thus,
by definition of $L_{1}$, we see that $\<L_{1}Q, h_{1}(\cdot+y)\>=0$. Therefore, inequality above shows
\[
\|h\|^{2}_{H^{1}}-\gamma|u(t,0)|^{2}\lesssim \lambda^{2}+o(\|h\|^{2}_{H^{1}}).
\]
Thus,
\begin{equation}\label{BoundV}
\|h\|^{2}_{H^{1}}\lesssim \lambda^{2}
\quad \text{and}\quad 
-\gamma|u(t,0)|^{2}\lesssim \lambda^{2}.
\end{equation}
On the other hand, as $M(Q)=M(u)$ \eqref{Defg} implies
\begin{equation}\label{ExpanE}
\lambda^{2}\|Q\|^{2}_{L^{2}}+2\lambda\|Q\|^{2}_{L^{2}}+2\<Q(\cdot-y), h_{1}\>+\|h\|^{2}_{L^{2}}=0.
\end{equation}
We observe that by \eqref{BoundR}, \eqref{BoundV} and \eqref{ExpanE} we obtain
\begin{equation}\label{BoundH}
|\lambda| \lesssim \|h\|_{H^{1}},
\end{equation}
so that 
\begin{align}\label{aproxG}
&\|g\|_{H^{1}}\sim |\lambda|	\sim \|h\|_{H^{1}},\\ \label{aproxProduct}
&\<Q(\cdot-y), h_{1}\>=-\lambda\|Q\|^{2}_{L^{2}}+\O(\lambda^{2}).
\end{align}
Finally, since $\<\partial_{x}Q(\cdot-y), \partial_{x}h_{1} \>=-\<Q(\cdot-y), h_{1} \>$,  
combining \eqref{BoundV}, \eqref{aproxG} and 
\eqref{aproxProduct} we get
\begin{align*}
	\mu(t)&=\|\partial_{x}Q\|^{2}_{L^{2}}-\|\partial_{x}[(1+\lambda)Q(\cdot-y)+h]\|^{2}_{L^{2}}+
	\gamma|u(t,0)|^{2}\\
	&=-4\lambda \|\partial_{x}Q\|^{2}_{L^{2}}+\O(\lambda^{2}),
\end{align*}
which proves that $\mu \sim |\lambda|$. Hence $\mu(t) \sim \|g(t)\|_{H^{1}}$.
This completes the proof of lemma.
\end{proof}

\begin{lemma}\label{BoundEx}
Under the conditions of Lemma \ref{BoundI}, if $\mu_{0}$ is sufficiently small, then
\begin{equation}\label{EquaEx22}
\frac{e^{-2|y(t)|}}{|y(t)|^{2}}\lesssim\mu(t) \quad \text{for all $t\in I_{0}$}.
\end{equation}
\end{lemma}
\begin{proof}
Using \eqref{Defg} yields that
\[
|u(t, 0)|^{2}=|Q(y(t))|^{2}+(\lambda^{2}+2\lambda)|Q(y(t))|^{2}+|h(t,0)|^{2}+2Q(y(t))h_{1}(t,0).
\]
Thus, by \eqref{BoundV} and \eqref{aproxG}  we obtain that
\begin{equation}\label{BoundQ}
|Q(y(t))|^{2}\lesssim [\mu(t)]^{2}+\mu(t)\lesssim\mu(t).
\end{equation}
On the other hand, it is well known that 
\[
|Q(x)|\gtrsim |x|^{-1}e^{-|x|}\quad \text{for $|x|\geq 1$}.
\]
From \eqref{BoundQ} and according Remark \ref{Yinfinity} we have that if $\mu_{0}$ is sufficiently small, then
\[
\frac{e^{-2|y(t)|}}{|y(t)|^{2}}\lesssim |Q(y(t))|^{2}\lesssim \mu(t), \quad \text{for all $t\in I_{0}$}.
\]
Thus, we obtain that \eqref{EquaEx22} is true.
\end{proof}

\begin{lemma}\label{BoundEx22}
Under the conditions of Lemma \ref{BoundI}, we have
\begin{equation}\label{EquaEx}
|y^{\prime}(t)|\lesssim \mu(t) \quad \text{for all $t\in I_{0}$}.
\end{equation}
\end{lemma}
\begin{proof}
We recall that $g(t)=e^{-i\theta(t)}[u(t)-e^{i\theta(t)}Q(\cdot-y(t))]$. 
Notice that
\begin{equation}\label{EquationG}
\begin{split}
i\partial_{t}g+\Delta g-\theta^{\prime}g
+Q(x-y)-\theta^{\prime}Q(x-y)-y^{\prime}\partial_{x}Q(x-y)\\
-\gamma\delta(x)(e^{-i\theta}u)+[f(g-Q(\cdot-y))-f(Q(\cdot-y))]=0,
\end{split}
\end{equation}
where $f(z)=|z|^{p-1}z$.
We claim that
\begin{equation}\label{Boundteta}
|\theta^{\prime}(t)|\lesssim 1+|y^{\prime}(t)|\|g(t)\|_{H^{1}}.
\end{equation}
Indeed, multiplying \eqref{EquationG} with $Q(\cdot-y)$, integrating on $\R$, using the orthogonality 
conditions \eqref{Ortogonality}, estimate \eqref{DeltaBound}
and taking the real part we have (recall that $\|g(t)\|_{H^{1}}\ll 1$)
\begin{equation}\label{QuasiBound}
|\theta^{\prime}(t)|\lesssim_{Q} |\RE\<i\partial_{t}g, Q(\cdot-y)\>|+|\theta^{\prime}(t)|\|g(t)\|_{H^{1}}+  1+\O(\|g(t)\|_{H^{1}}).
\end{equation}
Here we have used the inequality
\[
|f(g-Q(\cdot-y))-f(Q(\cdot-y))|\lesssim|Q(\cdot-y)|^{p-1}|g|+|g|^{p}.
\]
Next, since  (see \eqref{Ortogonality}) 
\[
\frac{d}{dt}\IM\< g(t), Q(\cdot-y(t)) \>=0
\]
we obtain that
\begin{equation}
\begin{split}
\label{FinalEstimate}
|\RE\<i\partial_{t}g, Q(\cdot-y) \>|=|\IM \<\partial_{t}g, Q(\cdot-y) \>|
=|y^{\prime}(t)\IM\< g(t), \partial_{x}Q(\cdot-y) \>|\\
\lesssim |y^{\prime}(t)|\|g(t)\|_{H^{1}}
\end{split}
\end{equation}
By using the fact that $\|g(t)\|_{H^{1}}\ll 1$, combining \eqref{QuasiBound} and \eqref{FinalEstimate} we get \eqref{Boundteta}.

Now we obtain \eqref{EquaEx}. Indeed, multiplying \eqref{EquationG} by $\partial_{x}Q(\cdot-y(t))$, integration over $\R$,
and using \eqref{DeltaBound} we see that 
\begin{equation}\label{Yprime11}
|y^{\prime}(t)|\lesssim |\IM\< i\partial_{t}g, \partial_{x}Q(\cdot-y) \>|+(1+|\theta^{\prime}(t)|)\|g(t)\|_{H^{1}}.
\end{equation}
Moreover, orthogonality conditions \eqref{Ortogonality} implies
\[
\begin{split}
|\IM\<i\partial_{t}g, \partial_{x}Q(\cdot-y) \>|=|\RE \<\partial_{t}g, \partial_{x}Q(\cdot-y) \>|
=|y^{\prime}(t)\RE\< g(t), \partial_{x}Q(\cdot-y) \>|\\
\lesssim |y^{\prime}(t)|\|g(t)\|_{H^{1}}
\end{split}
\]
As $\|g(t)\|_{H^{1}} \sim \mu(t)$, by inequality above, \eqref{Yprime11} and \eqref{Boundteta} we obtain \eqref{EquaEx}.
This completes the proof of lemma.
\end{proof}

\begin{proof}[{Proof of Proposition \ref{Modilation11}}]
Putting together Lemma \ref{ExistenceF} with the bounds \eqref{DeltaBound},  \eqref{EquaEx22} and \eqref{EquaEx}, the result
 quickly follows.
\end{proof}

\section{Precluding the compact solution}\label{S:Impossibility}

Throughout this section we assume that $u$ is the solution constructed in Proposition \ref{Criticalsolution}.
We recall that $u$ satisfies the following properties:
\[
E_{\gamma}(u_{0})=E_{0}(Q), \quad M(u_{0})=M(Q),
P_{\gamma}(u_{0})\geq 0
\quad \text{and}\quad 
\|u \|_{L_{t}^{a}L^{r}_{x}([0, \infty)\times\R)}=\infty.
\]
Moreover, 
\[
\text{$\left\{u(t, \cdot+x_{0}(t)):t\in [0, \infty)\right\}$ is pre-compact in $H^{1}(\R)$.}
\]

In this section, we will apply the localized Virial identities and Proposition \ref{Modilation11} (Modulation theory)
to show the impossibility of the 'compact' solution $u$  established above.

Before proving Theorem \ref{Th2}, some preliminaries are needed.
We follow mainly \cite{MiaMurphyZheng2021} here, and the proof for the next lemma is essentially the same as in \cite[Lemma 4.2]{MiaMurphyZheng2021}, thus, we omit it.

\begin{lemma}\label{Parametrization}
If $\mu_{0}$ is sufficiently small, then there exists $C>0$ such that
\[
|x_{0}(t)-y(t)|<C \quad \text{for  $t\in I_{0}$}.
\]
Here, the parameter $y(t)$ is given in Proposition \ref{Modilation11}.
\end{lemma}

As a consequence of the Lemma \ref{Parametrization} we obtain that
\begin{equation}\label{CompacNew}
\text{$\left\{u(t, \cdot+x(t))\right\}$ is pre-compact in $H^{1}(\R)$}
\end{equation}
where the spatial center is given by
\[x(t)=
\begin{cases}
x_{0}(t)& \quad t\in [0, \infty)\setminus I_{0},\\
y(t)&\quad t\in I_{0}.
\end{cases}
\]

\begin{lemma}\label{SequenceInf}
For any sequence $\left\{t_{n}\right\}\subset [0, \infty)$, we have
\begin{equation}\label{ZeroPoten}
|x(t_{n})|\to \infty \quad \text{if and only if}\quad |u(t_{n},0)|^{2}\to 0.
\end{equation}
\end{lemma}
\begin{proof}
Assume that  $|x(t_{n})|\to \infty$ and there exists $\epsilon>0$ such that $|u(t_{n},0)|\geq \epsilon$ for every $n\in \N$. Using \eqref{CompacNew} we see that there exists $f\in H^{1}(\R)$ such that, possibly for a subsequence only,
\[
\|u(t_{n})-f(\cdot-x(t_{n}))\|_{H^{1}}\to 0 \quad \text{as $n\to \infty$}.
\]
Therefore,
\[
\lim_{n\to\infty}|f(-x(t_{n}))|\geq \frac{\epsilon}{2},
\]
which is a contradiction because $|x(t_{n})|\to \infty$. This proves the first implication of lemma.

Now, assume that $ |u(t_{n},0)|^{2}\to 0$ but (possibly for a subsequence only) $\left\{x(t_{n})\right\}$ converges. 
By \eqref{CompacNew} we infer that there exists $f\in H^{1}(\R)$ such that
\begin{equation}\label{contra11}
u(t_{n})\to f\quad \text{in $H^{1}(\R)$ with }\quad |f(0)|=0.
\end{equation}
Next, since $ |u(t_{n},0)|^{2}\to 0$, by using \eqref{22condition} and \eqref{contra11} we get
\[
E_{0}(f)=E_{0}(Q) \quad \text{and}\quad M(f)=M(Q).
\]
But then  there exists $\theta$ and $x_{0}\in \R$ such that $f(x)=e^{i \theta}Q(x-x_{0})$.
In particular, $|f(0)|=Q(x_{0})>0$, which is a contraction with \eqref{contra11}.

\end{proof}

\begin{lemma}\label{DeltaZero}
If $t_{n}\to \infty$, then we have
\begin{equation}\label{equivalence}
|x(t_{n})|\to \infty \quad \text{if and only if}\quad \mu(t_{n})\to 0.
\end{equation}
\end{lemma}
\begin{proof}
First assume $\mu(t_{n})\to 0$. Then combining  Proposition \ref{Modilation11} and 
Lemma \ref{SequenceInf} we infer that $|x(t_{n})|\to \infty$.

On the other hand, let $t_{n}\to \infty$ and assume by contradiction that $|x(t_{n})|\to \infty$ but, 
possibly for a subsequence only,
\begin{equation}\label{Zeroplus}
\mu(u(t_{n}))\geq c>0.
\end{equation}
Since $\left\{u(t_{n}, \cdot+x(t_{n}))\right\}$ is pre-compact in $H^{1}(\R)$, there exist a subsequence (still denoted bi itself)
and $v_{0}\in H^{1}(\R)$ such that 
\begin{equation}\label{CompactV}
u(t_{n}, \cdot+x(t_{n})) \to v_{0} \quad \text{in}\quad H^{1}(\R).
\end{equation}
In particular, from \eqref{22condition}, \eqref{ZeroPoten} and \eqref{Zeroplus}  we have
\[
M(v_{0})=M(Q), \quad E_{0}(v_{0})=E_{0}(Q) \quad \text{and}\quad \|\partial_{x}v_{0}\|^{2}_{L^{2}}<\|\partial_{x}Q\|^{2}_{L^{2}}.
\]
Let $v$ be the solution of the free NLS on the real line $\R$ (i.e., \eqref{NLS} with $\gamma=0$) with initial data $v_{0}$. As consequence of \cite[Theorem 1.3]{CamposFarahRou2020} we have that
the solution $v$ is global well-posedness. Moreover, either scatters as $t\to \infty$ or as $t\to -\infty$ (or both).

Suppose that $v$ scatters as $t \to \infty$. Let $\widetilde{v}_{n} :=v(t,x-x(t_n))$. Since $|x(t_{n})|\to \infty$ and
$\|v\|_{L_t^aL_x^r([0,\infty)\times\R)} \lesssim 1$, it follows from  \cite[Proposition 3.4]{BaniVisci2016} that 
\begin{align*}
	\widetilde{v}_n (t,x) = e^{-itH_{\gamma}} v_0(\cdot -x(t_n)) + i \int_{0}^{t} e^{-i(t-s)H_{\gamma}} |\widetilde{v}_n |^{p-1}\widetilde{v}_n ds + g_n(t,x)
\end{align*}
where $\|g_n\|_{L_t^a L_x^r([0,\infty)\times\R)} \to 0$ as $t \to \infty$. We have $\| u(t_n) - \widetilde{v}_n\|_{H^1} \to 0$ as $t\to \infty$ by \eqref{CompactV}. By Lemma \ref{perturb}, we obtain $\|u(t_n+t)\|_{L_t^a L_x^r([0,\infty)\times\R)}
 = \|u\|_{L_t^a L_x^r([t_n,\infty)\times\R)} \lesssim 1$ for large $n$. This contradicts that $u$ does not scatter in positive time direction. 
Suppose that $v$ scatters as $t \to -\infty$. By the similar argument, we obtain 
\[\|u(t_n+t)\|_{L_t^a L_x^r((-\infty,0]\times\R)} = \|u\|_{L_t^a L_x^r((-\infty,t_n]\times\R)} \lesssim 1
\] for large $n$. This also contradicts that $u$ does not scatter in positive time direction. 
As a consequence, $\mu(u(t_n)) \to 0$ as $n \to \infty$. 

\end{proof}

Recall that $I_{0,\infty}$ was defined in Lemma \ref{VirialIden}.

\begin{lemma}\label{Compa}
Fix $\gamma<0$. There exists $c>0$ such that
\begin{equation}\label{InequeN}
I_{0,\infty}[u(t)]=\|\partial_{x}u(t)\|^{2}_{L^{2}}-\frac{(p-1)}{2(p+1)} \|u(t)\|^{p+1}_{L^{p+1}}\geq c\mu(t)
\end{equation}
\end{lemma}
\begin{proof}
Assume that \eqref{InequeN} is not true. Then there exists a sequence $\left\{t_{n}\right\}$ such that
\begin{equation}\label{Contra22}
I_{0,\infty}[u(t_{n})]\leq \frac{1}{n}\mu(t_{n}).
\end{equation}
Using Lemma \ref{GlobalS} we infer that the sequence $\left\{\mu(t_{n})\right\}$ is bounded. We claim that
\[
\mu(t_{n})\to 0 \quad \text{as $n\to\infty$}.
\]
Indeed, by using the Pohozaev's identities we have (see Section \ref{S:preli})
\begin{equation}\label{Pohi22}
\|\partial_{x}Q\|^{2}_{L^{2}}=\frac{(p-1)}{2(p+1)}\|Q\|^{p+1}_{L^{p+1}}
\quad \text{and}\quad
\frac{p-1}{2}E_{0}(Q)=\frac{(p-5)}{4}\|\partial_{x}Q\|^{2}_{L^{2}}.
\end{equation}
Combining this inequalities we can write
\begin{equation}\label{Aux11}
I_{0,\infty}[u(t_{n})]
=\(\frac{p-5}{4}\)\mu(t_{n})+\gamma|u(t_{n}, 0)|^{2}.
\end{equation}
Next, notice that $I_{0,\infty}[u(t_{n})]>0$. Indeed, by using sharp Gagliardo-Nirenberg inequality \eqref{GI}, \eqref{Su},  \eqref{Gradient} and \eqref{Pohi22} we get
\[
\|u(t_{n})\|^{p+1}_{L^{p+1}}\leq 
\frac{2(p+1)}{p-1}\(\frac{\|\partial_{x}u(t_{n})\|_{L^{2}}}{\|\partial_{x}Q\|_{L^{2}}}\)^{\frac{p-5}{2}}
\|\partial_{x}u(t_{n})\|^{2}_{L^{2}}<\frac{2(p+1)}{p-1}\|\partial_{x}u(t_{n})\|^{2}_{L^{2}},
\]
which implies that $I_{0,\infty}[u(t_{n})]>0$. Thus, by \eqref{Su} and \eqref{Contra22} we infer that
\begin{equation}\label{Aux33}
S_{0}(u(t_{n}))\leq S_{0}(Q)\quad \text{and}\quad I_{0,\infty}[u(t_{n})]\to 0 \quad \text{as $n\to \infty$}. 
\end{equation}
That is, $\left\{u(t_{n}\right\}$ is a minimizing sequence of the variational problem \eqref{dinf}. In particular,
we have
\[
\|\partial_{x}u(t_{n})\|_{L^{2}}\to \|\partial_{x}Q\|_{L^{2}},
\]
which implies
\begin{equation}\label{Secondelta}
\mu(t_{n})=\gamma|u(t_{n},0)|^{2}+o(1) \quad \text{as $n\to \infty$}.
\end{equation}
Combining \eqref{Aux11} and \eqref{Secondelta} we obtain the claim.
Finally, by Proposition \ref{Modilation11} we have
\[
-\gamma|u(t_{n},0)|^{2}\lesssim \mu(t_{n})^{2}\leq \frac{(p-5)}{8} \mu(t_{n})\quad \text{for $n$ large}.
\]
Thus, by using \eqref{Aux11} and \eqref{Contra22} we get
\[
\frac{(p-5)}{8}\mu(t_{n})\leq \frac{1}{n}\mu(t_{n})\quad \text{for $n$ large},
\]
which is a contradiction with \eqref{deltapositive}.
\end{proof}

 In the next two propositions we study the behavior of the parameter $x(t)$ defined in \eqref{CompacNew}.

\subsection{If $x(t)$ is bounded, then $x(t)$ is unbounded}\label{Sub:11}

In this subsection we establish the following result.
\begin{proposition}\label{BoundedUN}
If the spacial center $x(t)$ is bounded, then $x(t)$ is unbounded.
\end{proposition}
\begin{proof}
The proof of proposition runs in two steps. 

\textbf{Step 1.} Virial estimate. Let $T>0$ and $\epsilon>0$, then there exists a constant $\rho_{\epsilon}=\rho(\epsilon)>0$ such that
\begin{equation}\label{NewVirial}
\frac{1}{T}\int^{T}_{0}\mu(t)dt\lesssim\epsilon
+\frac{1}{T}[\rho_{\epsilon}+\sup_{t\in [0,T]}|x(t)|]\|u\|_{L^{\infty}_{t}H^{1}_{x}}^{2}.
\end{equation}
Indeed, for $R>1$, which will be determined later, we can write (see Lemma \ref{VirialIden})
\begin{equation}\label{DecompV}
\frac{d}{dt}V_{R}[u(t)]=I_{\infty, 0}[u(t)]+(I_{R, \gamma}[u(t)]-I_{\infty, 0}[u(t)]).
\end{equation}
Notice that
\begin{align}\nonumber
&I_{R, \gamma}[u(t)]-I_{\infty, 0}[u(t)]\\\label{CViral11}
&=\int_{|x|> R}(-8)|\partial_{x}u(t,x)|^{2}+4\frac{p-1}{p+1}|u(t,x)|^{2}
+4[\partial^{2}_{x}w_{R}(x)]|\partial_{x}u(t,x)|^{p+1}dx\\\label{CViral22}
&+\int_{|x|> R}[-\partial^{4}_{x}w_{R}(x)]|u(t,x)|^{2}-2\frac{p-1}{p+1}[\partial^{2}_{x}w_{R}(x)]|u(t,x)|^{p+1}dx\\\label{CViral33}
&-4\gamma|u(t,0)|^{2}.
\end{align}
As $\gamma<0$, by \eqref{DecompV} and \eqref{InequeN} we get
\begin{equation}\label{NewDecomp}
c\delta(t)\leq \frac{d}{dt}V_{R}[u(t)]+|\eqref{CViral11}|+|\eqref{CViral22}|.
\end{equation}
By \eqref{CompacNew} we have uniform localization of $u$: for each $\epsilon>0$, there exists $\rho_{\epsilon}=\rho (\epsilon)>0$ independent of $t$ such that
\[
\int_{|x-x(t)|>\rho_{\epsilon}}|\partial_{x} u(t,x)|^{2}+|u(t,x)|^{p+1}+|u(t,x)|^{2}dx<\epsilon.
\]
Given $T>0$, set
\[
R:=\rho_{\epsilon}+\sup_{t\in [0,T]}|x(t)|.
\]
It follows easily that $\left\{|x|\geq R\right\}\subset \left\{|x-x(t)|\geq \rho_{\epsilon}\right\}$ for all $t\in [0, T]$,
thus by definition of $w_{R}$ (see Section \ref{S:preli}) 
\begin{equation}\label{smallI}
|\eqref{CViral11}|+|\eqref{CViral22}|<\epsilon
\end{equation}
for all $t\in [0, T]$. Moreover, since $|V_{R}[u(t)]|\lesssim R\|u\|^{2}_{L^{\infty}_{t}H^{1}_{x}}$, 
integrating \eqref{NewDecomp} on $[0, T]$ and applying \eqref{smallI}  yields
\[
\int^{T}_{0}\mu(t)\lesssim [\rho_{\epsilon}+\sup_{t\in [0,T]}|x(t)|]\|u\|^{2}_{L^{\infty}_{t}H^{1}_{x}}
+\epsilon T,
\]
which implies \eqref{NewVirial}.

\textbf{Step 2.} Conclusion.
The result is proved by contradiction. Assume that $x(t)$ is bounded. By \eqref{NewVirial} we have
\[
\frac{1}{T}\int^{T}_{0}\mu(t)dt\lesssim\epsilon
+\frac{1}{T}\rho_{\epsilon}\quad \text{for all $T>0$}, 
\]
and for any $\epsilon>0$. Consider a sequence $\epsilon_{n}\to 0$ as $n\to \infty$. By using the mean value theorem for integrals, and 
choose appropriately  times $T_{n}\to \infty$, we infer that there exists (recall that $\mu(t)>0$) a sequence $t_{n}\to \infty$ 
so that $\mu(t_{n})\to 0$ as $n\to \infty$, which together with Lemma \ref{DeltaZero} yields a contradiction. 
\end{proof}

\subsection{If $x(t)$ is unbounded, then $x(t)$ is bounded}\label{Sub:22}

Now our goal  is to prove the following result.
\begin{proposition}\label{NaoBounded}
If the spacial center $x(t)$ is unbounded, then $x(t)$ is bounded.
\end{proposition}

Before giving the proof of Proposition \ref{NaoBounded} some preparation is necessary.
 
\begin{lemma}\label{Lemma11}
Let $\mu_{1}\in (0, \mu_{0})$ be sufficiently small. There exists a constant $C=C(\mu_{1})>0$ such that for any interval $[t_{1}, t_{2}]\subset [0, \infty)$ we have
\begin{equation}\label{BoundT1}
\int^{t_{2}}_{t_{1}}\mu(t)dt\leq C\big[1+\sup_{t\in[t_{1},t_{2}]}|x(t)|\big]\left\{\mu(t_{1})+\mu(t_{2})\right\}.
\end{equation}
\end{lemma}
\begin{proof}
Consider $R>1$, which will be determined later. We use the localized virial identities in Lemma \ref{VirialModulate}
with the function $\chi(t)$ satisfying
\[
\chi(t)=
\begin{cases}
1& \quad \mu(t)<\mu_{1} \\
0& \quad \mu(t)\geq \mu_{1}.
\end{cases}
\]
By \eqref{InequeN} and Lemma \ref{VirialModulate} we have
\begin{equation}\label{VirilaX}
\frac{d}{dt}V_{R}[u(t)]=I_{\infty, 0}[u(t)]+\EE(t)\geq c\mu(t)+\EE(t)
\end{equation}
with
\begin{equation}\label{Error11}
\EE(t)=
\begin{cases}
I_{R, \gamma}[u(t)]-I_{\infty, 0}[u(t)]& \quad  \text{if $\mu(t)\geq \mu_{1}$}, \\
I_{R, \gamma}[u(t)]-I_{\infty, 0}[u(t)]-\K[u(t)]& \quad \text{if $\mu(t)< \mu_{1}$},
\end{cases}
\end{equation}
where
\begin{equation}\label{Error22}
\K(t)=I_{R,0}[e^{i\theta(t)}Q(\cdot-y(t))]-I_{\infty,0}[e^{i\theta(t)}Q(\cdot-y(t))].
\end{equation}
We assume the following two claims for a moment to conclude the proof.

\textit{{Claim I.}} Given $R>1$  we have
\begin{align}\label{EstimateV11}
	|V_{R}[u(t_{j})]|\lesssim \frac{R}{\mu_{1}}\mu(t_{j}) \quad& \text{if $\mu(t_{j})\geq \mu_{1}$ for $j=1$, $2$},\\
	\label{EstimateV22}
	|V_{R}[u(t_{j})]|\lesssim R \mu(t_{j}) \quad &\text{if $\mu(t_{j})< \mu_{1}$ $j=1$, $2$}.
\end{align}

\textit{{Claim II.}} Given $\epsilon>0$, there exists $\rho_{\epsilon}=\rho(\epsilon)>0$ such that if
$R=\rho_{\epsilon}+\sup_{t\in [t_{1}, t_{2}]}|x(t)|$ we have
\begin{align}\label{EstimateE11}
	\EE(t)\geq -\frac{\epsilon}{\mu_{1}}\mu(t) \quad &\text{uniformly for $t\in [t_{1}, t_{2}]$ and $\mu(t)\geq \mu_{1}$},\\
	\label{EstimateE22}
|\EE(t)|\leq \epsilon\mu(t)+\mu(t)^{2} \quad &\text{uniformly for $t\in [t_{1}, t_{2}]$ and $\mu(t)< \mu_{1}$}.
\end{align}

Integrating \eqref{VirilaX} on $[t_{1}, t_{2}]$ and putting all the estimates \eqref{EstimateV11}, \eqref{EstimateV22}, \eqref{EstimateE11} 
and \eqref{EstimateE22} together  yields
\[
\int^{t_{2}}_{t_{1}}\mu(t)dt\lesssim 
\frac{1}{\mu_{1}}\Big[\rho_{\epsilon}+\sup_{t\in [t_{1}, t_{2}]}|x(t)|\Big](\mu(t_{1})+\mu(t_{2}))
+\Big(\frac{\epsilon}{\mu_{1}}+\epsilon+\mu_{1}\Big)\int^{t_{2}}_{t_{1}}\mu(t)dt.
\]
Here we have used that $\mu(t)^{2}\lesssim \mu_{1}\mu(t)$.  Choosing $\epsilon=\epsilon(\mu_{1})$ sufficiently small
we get \eqref{BoundT1}.

Therefore, it remains to establish the above claims.
\begin{proof}[{Proof of Claim I}] First, assume that $\mu(t_{j})\geq \mu_{1}$. Then we have
\[
|V_{R}[u(t)]|=\left|2\IM\int_{\R}\partial_{x}[w_{R}(x)]\overline{u(t,x)}\partial_{x}u(t,x)dx\right|
\lesssim R\|u\|^{2}_{L^{\infty}_{t}H^{1}_{x}}\lesssim_{Q} \frac{R}{\mu_{1}}\mu(t_{j})
\]
which implies \eqref{EstimateV11}. Here, we have used that $\|u\|^{2}_{L^{\infty}_{t}H^{1}_{x}} \sim S_{0}(Q)$.
On the other hand, if $\mu(t_{j})< \mu_{1}$, by using the fact that $Q$ is real we get
\begin{align*}
|V_{R}[u(t)]|&=\left|2\IM\int_{\R}\partial_{x}w_{R}[\overline{u}\partial_{x}u-e^{-i\theta(t_{j})}Q(\cdot-y(t_{j}))
\partial_{x}(e^{i\theta(t_{j})}Q(\cdot-y(t_{j})))dx\right|\\
&\lesssim
R[\|u\|_{L^{\infty}_{t}H^{1}_{x}}+\|Q\|_{H^{1}}]
\|u(t_{j})-e^{i\theta (t_{j})}Q(\cdot-y(t_{j}))\|_{H^{1}}\\
&\lesssim_{Q}R\|g(t_{j})\|_{H^{1}}
\lesssim_{Q}R \mu(t_{j}).
\end{align*}
where in the last inequality we have used estimate \eqref{Estimatemodu} in  Proposition \eqref{Modilation11} 
(recall that $\mu(t_{j})< \mu_{1}<\mu_{0}$). This completes the proof of Claim I.
\end{proof}

\begin{proof}[{Proof of Claim II}]
First, assume that $\mu(t)\geq \mu_{1}$. By using \eqref{CompacNew} we obtain that for each $\epsilon>0$, there exists $\rho_{\epsilon}=\rho (\epsilon)>0$  such that
\begin{equation}\label{CompactAgain}
\int_{|x-x(t)|>\rho_{\epsilon}}|\partial_{x} u(t,x)|^{2}+|u(t,x)|^{p+1}+|u(t,x)|^{2}dx<\epsilon.
\end{equation}
We put
\[
R:=\rho_{\epsilon}+\sup_{t\in [t_{1},t_{2}]}|x(t)|.
\]
Notice  that $\left\{|x|\geq R\right\}\subset \left\{|x-x(t)|\geq \rho_{\epsilon}\right\}$ 
for all $t\in [t_{1}, t_{2}]$. Since $\gamma<0$, using the same argument developed above 
(see \eqref{CViral11}-\eqref{CViral33} and \eqref{smallI}) 
implies
\[
 I_{R, \gamma}[u(t)]-I_{\infty, 0}[u(t)]\geq -\epsilon\geq -\frac{\epsilon}{\mu_{1}}\mu(t)
\quad \text{for every $t\in [t_{1}, t_{2}]$ with $\mu(t)\geq \mu_{1}$}.
\]
Thus we have \eqref{EstimateE11}.

Next, suppose  $\mu(t)< \delta_{1}$. In order to simplify the notation, we put $Q(t)=e^{i\theta(t)}Q(\cdot-y(t))$.
We also recall that $g(t)=e^{-i\theta(t)}[u(t)-Q(t)]$.
By definition of $\EE(t)$, given in \eqref{Error11}, we can write for $t\in [t_{1}, t_{2}]$ (see \eqref{CViral11}-\eqref{CViral33})
\begin{align}\label{Decomp11}
\EE(t)&=\int_{|x|> R}(-8)[|\partial_{x}u(t)|^{2}-|\partial_{x}Q(t)|^{2}]
+4\frac{p-1}{p+1}[|u(t)|^{p+1}-|Q(t)|^{p+1}]\\\label{Decomp12}
&+\int_{|x|> R}[-\partial^{4}_{x}w_{R}][|u(t)|^{2}-|Q(t)|^{2}]
-2\frac{p-1}{p+1}[\partial^{2}_{x}w_{R}][|u(t)|^{p+1}-|Q(t)|^{p+1}]dx\\\label{Decomp33}
&-4\gamma|u(t,0)|^{2}.
\end{align}
By the elemental inequality
\[
||z_{1}|^{\alpha+1}-|z_{2}|^{\alpha+1}|\lesssim|z_{1}-z_{2}|(|z_{1}|^{\alpha}+|z_{2}|^{\alpha})
\quad \text{for any $z_{1}$, $z_{2}\in \C$ and $\alpha>0$},
\]
we get
\begin{equation*}
|\eqref{Decomp11}|+|\eqref{Decomp12}|
\lesssim 
[\|u(t)\|^{\alpha}_{H_{x}^{1}(|x|\geq R)}+\|Q(\cdot-y(t))\|^{\alpha}_{H_{x}^{1}(|x|\geq R)}]
\|g(t)\|_{H_{x}^{1}}
\end{equation*}
with $\alpha\in\left\{1,p\right\}$. Then by Proposition \ref{Modilation11} (note that $x(t)=y(t)$ because $\mu(t)<\mu_{1}$) and 
\eqref{CompactAgain} we get 
\begin{equation}\label{PrimerIne}
|\eqref{Decomp11}|+|\eqref{Decomp12}|
\lesssim
[\|u(t)\|^{\alpha}_{H_{x}^{1}(|x-x(t)|\geq \rho_{\epsilon})}+\|Q\|^{\alpha}_{H_{x}^{1}(|x|\geq \rho_{\epsilon})}] \mu(t)
\lesssim_{Q}
\epsilon \mu(t).
\end{equation}
Finally, since $\mu(t)<\mu_{1}$, again estimate \eqref{Estimatemodu} in Proposition \ref{Modilation11} implies
\begin{equation}\label{Secondine}
\eqref{Decomp33}\lesssim 
 \mu(t)^{2}
\end{equation}
Thus, combining \eqref{PrimerIne} and \eqref{Secondine} we obtain \eqref{EstimateE22}. This completes the proof of Claim II.
\end{proof}
\end{proof}

\begin{proposition}\label{Spatialcenter}
Let $[t_{1}, t_{2}]$ be an interval of $[0, \infty)$. Then there exists $\mu_{1}>0$ and $C>0$ such that
\begin{equation}\label{BoundCenter}
|x(t_{1})-x(t_{2})|\leq \frac{C}{\mu_{1}}\int^{t_{2}}_{t_{1}}\mu(t)+2C.
\end{equation}
\end{proposition}
\begin{proof}
 The proof is divided into three steps.

\textsl{Step 1.} There exists a constant $C$ such that
\begin{equation}\label{step11}
|x(t)-x(s)|\leq C \quad \text{for all $t$, $s\geq 0$ such that $|t-s|\leq 1$}.
\end{equation}
 The proof of \eqref{step11} is the same (with obvious modifications) as the one given in \cite[Lemma 4.11]{MiaMurphyZheng2021} therefore we omit the details.

\textsl{Step 2.}
There exists $\mu_{1}>0$ such that either
\begin{equation}\label{MinMax}
\inf_{t\in [T, T+1]}\mu(t)\geq \mu_{1} \quad \text{or}\quad
\sup_{t\in [T, T+1]}\mu(t)<\mu_{0}\quad \text{for all $T\geq 0$}.
\end{equation}
Indeed, \eqref{MinMax} is proved by contradiction. Assume that there exist $t_{n}^{\ast}\geq 0$ and two sequences
$t_{n}$, $t^{\prime}_{n}\in  [t_{n}^{\ast}, t_{n}^{\ast}+1]$ such that, possibly for a subsequence only,
\begin{align}\label{ContraStep2}
&\mu(t_{n})\to 0 \quad \text{and}\quad \mu(t^{\prime}_{n})\geq \mu_{0} \quad \text{as $n\to \infty$},\\
\label{ContraLimit}
& t_{n}-t^{\prime}_{n}\to t^{\ast}\in[-1,1].
\end{align}
By Lemma \ref{DeltaZero} and Step 1 we infer that $|x(t^{\prime}_{n})|$ is bounded.  Thus, by \eqref{CompacNew} we obtain that there exits
$\varphi\in H^{1}(\R)$ such that 
\begin{equation}\label{Step2Conver}
\text{$u(t^{\prime}_{n})\to \varphi$ strongly in $H^{1}(\R)$ as $n\to \infty$}.
\end{equation}
Next, we show that $\mu(t_{n}^{\prime}+t^{\ast})=\mu(u(t_{n}^{\prime}+t^{\ast}))\to 0$ as $n\to \infty$. 
Indeed, 
the trivial estimate
\[
\|\int^{t_{2}}_{t_{1}}e^{isH_{\gamma}}F(s)ds\|_{H^{1}}\leq \|F(s)\|_{L^{1}_{t}H_{x}^{1}([t_{1}, t_{2}]\times \R)}
\]
implies (by Sobolev embedding)
\[\begin{split}
\|u(t_{n})-u(t_{n}^{\prime}+t^{\ast})\|_{H_{x}^{1}}
&\lesssim \||u(s)|^{p-1}u(s)\|_{L^{1}_{t}H_{x}^{1}([t_{n}, t_{n}^{\prime}+t^{\ast}]\times \R)}\\
&\lesssim \int^{t_{n}^{\prime}+t^{\ast}}_{t_{n}}\|u\|^{p-1}_{L^{\infty}_{x}}\|u\|_{H^{1}_{x}}ds
\lesssim_{Q} |(t_{n}^{\prime}-t_{n})+t^{\ast}|\to 0
\end{split}
\]
as $n\to \infty$. Here we have used that $\|u\|_{L^{\infty}_{t}H^{1}_{x}(\R\times\R)} \lesssim S_{0}(Q)$ (see Lemma \ref{GlobalS}).
Therefore, by continuity of functional $\mu(t)$ and \eqref{ContraStep2} we get  
\begin{equation}\label{Contradeltazero}
\delta(t_{n}^{\prime}+t^{\ast})\to 0 \quad \text{as $n\to \infty$}.
\end{equation}

Next, we show that there exists a positive constant $k$ such that $\mu(t_{n}^{\prime}+t^{\ast})\geq k$,
which is a contradiction with \eqref{Contradeltazero} and finished the proof. Indeed, by using \eqref{CompacNew}
we see that $M(\varphi)=M(Q)$, $E_{\gamma}(\varphi)=M(Q)$ and $P_{\gamma}(\varphi)\geq 0$. Let $v(t)$ the maximal solution
of \eqref{NLS} corresponding to the initial value $\varphi$. By Lemma \ref{GlobalS} we have that $v(t)$ is a global solution 
with $\mu(v(t))>0$ for all $t\in \R$. Then, since $t^{\ast}\in [-1, 1]$, by \eqref{Step2Conver} and continuous dependence (see\cite[Chapter 4]{CB}) we obtain  $u(t_{n}^{\prime}+t^{\ast})\to v(t^{\ast})$ for $n$ large enough. In particular, by continuity of $\mu(t)$
we infer that $\mu(t_{n}^{\prime}+t^{\ast})\geq k>0$ for $n$ large enough. This completes the proof of Step 2.

\textsl{Step 3.} Conclusion.
With Steps 1 and 2 and Proposition \ref{Modilation11} in hand, the proof of Proposition \ref{Spatialcenter} is the same as that of \cite[Proposition 4.10]{MiaMurphyZheng2021}.
\end{proof}

\begin{proof}[{Proof of Proposition \ref{NaoBounded}}]
The proof of Proposition \ref{NaoBounded} is almost the same as that given in \cite[Proposition 4.8]{MiaMurphyZheng2021}. For the sake of 
completeness, we repeat the argument in \cite{MiaMurphyZheng2021}.

By hypothesis we have that $|x(t)|$ is unbounded. Thus, we can pick a sequence $t_{n}\to \infty$ such that
\[
|x(t_{n})|=\sup_{t\in [0, t_{n}]}|x(t)|
\quad \text{with}\quad 
|x(t_{n})|\to\infty\quad \text{as $n\to \infty$}.
\]
Choosing $N\in \N$ large enough, we can apply Lemma \ref{DeltaZero} to get
\[
\mu(t_{n})<\frac{\mu_{1}}{100CC_{\mu_{1}}}\quad \text{for all $n\geq N$}.
\]
Here $\mu_{1}$ and  the constant $C$
is as in Proposition \ref{Spatialcenter} and $C_{\mu_{1}}=C(\mu_{1})$ is as in Lemma \ref{Lemma11}.
Consequently, by \eqref{BoundT1} and \eqref{BoundCenter}  we obtain
\begin{align*}
	|x(t_{n})-x(t_{N})|&\leq \frac{C}{\mu_{1}}\int^{t_{n}}_{t_{N}}\mu(t)dt+2C\\
	&\leq K+\frac{CC_{\mu_{1}}}{\mu_{1}}|x(t_{n})|(\mu(t_{n})+\mu(t_{N}))\\
	&\leq K+\frac{1}{2}|x(t_{n})|,
	\end{align*}
	where $K$ is a positive constant. This implies that for $n\geq N$
	\[
	|x(t_{n})|\leq |x(t_{N})|+2K,
	\]
	which is a contradiction. This completes the proof of the proposition.
\end{proof}

Now we are ready to give the proof of the Theorem \ref{Th2}.

\begin{proof}[{Proof of Theorem \ref{Th2}}]
Assume that Theorem \ref{Th2} fails, then  Proposition \ref{Th2} implies that there exists a element 
$u_{0}\in H^{1}(\R)$ and a spatial center $x(t)$ such that the corresponding solution to equation \eqref{NLS} verifies that  
$\left\{u(t, \cdot+ x(t)): t\geq 0\right\}$ is precompact in $H^{1}(\R)$. Then, combining Propositions \ref{BoundedUN} and \ref{NaoBounded} we get a contradiction.
\end{proof}

\section{Proof of Theorem \ref{BoundTheorem}}\label{BoundL}
In this section, we prove Theorem \ref{BoundTheorem}. We follow closely the proof of Theorem 1.5 in \cite{KillipMurphyVZ}.

\begin{proof}[{Proof of Theorem \ref{BoundTheorem}}]
Let $\varphi_{n}=(1-\epsilon_{n})Q(x-x_{n})$ with $\epsilon_{n}\to 0$ and $|x_{n}|\to \infty$. 
Since $|Q(x_{n})|\to 0$ as $n\to \infty$, we infer that
\[E_{\gamma}(\varphi_{n})[M(\varphi_{n})]^{\sigma}\nearrow  E_{0}(Q)[M(Q)]^{\sigma}
\quad \text{and}\quad 
P_{\gamma}(\varphi_{n})\to 0
\]
as $n\to \infty$. By using the fact that $|Q(x_{n})|> 0$ for all $n\in \N$ and $\gamma<0$ we have $P_{\gamma}(\varphi_{n})>0$
for all $n\in \N$. Thus, Theorem \ref{Th1} implies that the corresponding solution $u_{n}$ to \eqref{NLS} with initial data $\varphi_{n}$
exists globally and scatters.
We set
\[
\tilde{v}_{n}(t,x)=(1-\epsilon_{n})e^{it}Q(x-x_{n}).
\]
Let $T>0$ fixed. We want to apply Lemma \ref{perturb} (Long time perturbation) over $[-T,T]\times\R$, thus we need to estimate $\|e_{n}\|_{L^{a}_{t}L_{x}^{r}([-T,T]\times \R)}$,
where
\[
e_{n}(t,x)=\tilde{v}_{n}(t,x)-\left[(1-\epsilon_{n})e^{-itH_{\gamma}}Q(x-x_{n})+i\int^{t}_{0}e^{-i(t-s)H_{\gamma}}(|\tilde{v}_{n}(s)|^{p-1}\tilde{v}_{n}(s))ds \right].
\]
Indeed, as $R(t,x)=e^{it}Q(x)$ is a solution of the free NLS (i.e., \eqref{NLS} with $\gamma=0$) and $|x_{n}|\to \infty$,  using the same argument developed in the proof of Proposition 3.4  in \cite{BaniVisci2016}, we obtain 
\begin{equation}\label{AproxR}
R_{n}(t,x)=e^{-itH_{\gamma}}Q(x-x_{n})+i\int^{t}_{0}e^{-i(t-s)H_{\gamma}}(|R_{n}(s)|^{p-1}R_{n}(s))ds+g_{n}(t,x),
\end{equation}
where $(t,x)\in [-T,T]\times\R$ and $\|g_{n}(t,x)\|_{L^{a}_{t}L_{x}^{r}([-T,T]\times \R)}\to 0$ as $n\to \infty$. Here $R_{n}(t,x)=e^{it}Q(x-x_{n})$. 
By \eqref{AproxR} we have that
\begin{align*}
e_{n}(t,x)&=[\tilde{v}_{n}(t,x)-R_{n}(t,x)]+\epsilon_{n}e^{-itH_{\gamma}}Q(x-x_{n})\\
&+(1-(1-\epsilon)^{p})i\int^{t}_{0}e^{-i(t-s)H_{\gamma}}(|\tilde{v}_{n}(s)|^{p-1}\tilde{v}_{n}(s))ds-
g_{n}(t,x).
\end{align*}
Strichartz's estimates implies
\[
\|e_{n}\|_{L^{a}_{t}L_{x}^{r}([-T,T]\times \R)}\lesssim_{T, Q}[\epsilon_{n}+(1-(1-\epsilon)^{p})]
+\|g_{n}\|_{L^{a}_{t}L_{x}^{r}([-T,T]\times \R)}.
\]
Therefore, for any $T>0$ fixed we get $\|e_{n}\|_{L^{a}_{t}L_{x}^{r}([-T,T]\times \R)}\to 0$ as $n\to \infty$.
Finally, since $\|\tilde{v}_{n}\|_{L^{a}_{t}L_{x}^{r}([-T,T]\times \R)}\gtrsim_{Q}T$, it follows from Lemma \ref{perturb} that
\[
\|u_{n}\|_{L^{a}_{t}L_{x}^{r}([-T,T]\times \R)}\gtrsim_{Q}T,
\]
which finished the proof.
\end{proof}

\medskip

\noindent {\bf Acknowledgments}.
The second author is supported by JSPS KAKENHI Grant-in-Aid for Early-Career Scientists JP18K13444.


\begin{thebibliography}{10}

\bibitem{APIN}
S.~Albeverio, F.~Gesztesy, R.~H{\o}egh-Krohn and H.~Holden.
\newblock {\em Solvable Models in Quantum Mechanics}.
\newblock Springer-Verlag, New York, 1988.


\bibitem{ArCeGo}
A.H.~Ardila, L.~Cely and N.~Goloshchapova.
\newblock {\em Instability of ground states for the NLS equation with potential on the star graph}.
\newblock {\em To appear, J. Evol. Equ.}, 27 pages,  2021.



\bibitem{BaniVisci2016}
V.~Banica and N.~Visciglia.
\newblock Scattering for {NLS} with a delta potential.
\newblock {\em  J. Differential Eq.}, 260:4410--4439, 2016.

\bibitem{BMQT}
B.~Bellazi and M.~Mintchev.
\newblock Quantum field theory on star graphs.
\newblock {\em Phys. A: Math. Theor.}, 39:1101--1117, 2006.

\bibitem{APLD}
M.~Belloni and R.~W. Richard.
\newblock The infinite well and dirac delta function potentials as pedagogical,
  mathematical and physical models in quantum mechanics.
\newblock {\em Physics Reports}, 540(2):25--122, 2014.

\bibitem{CamposFarahRou2020}
L.~Campos, L.~G. Farah and S.~Roudenko.
\newblock Threshold solutions for the nonlinear {S}chr\"odinger equation.
\newblock {\em Preprint arXiv:2010.14434}, 2020.

\bibitem{AFI}
V.~Caudrelier, M.~Mintchev and E.~Ragoucy.
\newblock Solving the quantum nonlinear Schr\"odinger equation with
  $\delta$-type impurity.
\newblock {\em J. Math. Phys.}, 4(46):1--24, 2005.

\bibitem{CB}
T.~Cazenave.
\newblock {\em Semilinear Schr\"odinger Equations}.
\newblock Courant Lecture Notes in Mathematics,10. American Mathematical
  Society, Courant Institute of Mathematical Sciences, 2003.

\bibitem{DuyLandRoude}
T.~Duyckaerts, O.~Landoulsi and S.~Roudenko.
\newblock Threshold solutions in the focusing {3D} cubic NLS equation outside a
  strictly convex obstacle.
\newblock {\em Preprint arXiv:2010.07724}, 2020.

\bibitem{DuyckaertsRou2010}
T.~Duyckaerts and S.~Roudenko.
\newblock Threshold solutions for the focusing {3D} cubic Schr\"odinger
  equation.
\newblock {\em Rev. Mat. Iberoam.}, 26(1):1--56, 2010.

\bibitem{RJJ}
R.~Fukuizumi and L.~Jeanjean.
\newblock Stability of standing waves for a nonlinear Schr\"odinger
  equation with a repulsive {D}irac delta potential.
\newblock {\em Discrete Contin. Dyn. Syst.}, 21:121--136, 2008.

\bibitem{FO}
R.~Fukuizumi, M.~Ohta and T.~Ozawa.
\newblock Nonlinear Schr\"odinger equation with a point defect.
\newblock {\em Ann. Inst. H. Poincar\'e Anal. Non Lin\'eaire}, 25(5):837--845,
  2008.

\bibitem{GHW}
H.R. Goodman, P.J. Holmes and M.~Weinstein.
\newblock Strong {NLS} soliton--defect interactions.
\newblock {\em Physica D}, 192(3):215--248, 2004.

\bibitem{MT2}
M.~Ikeda and T.~Inui.
\newblock {G}lobal dynamics below the standing waves for the focusing
  semilinear Schr\"odinger equation with a repulsive {D}irac delta potential.
\newblock {\em Anal. PDE}, 10(2):481--512, 2017.


\bibitem{KillipMurphyVZ}
R. Killip, J. Murphy, M. Visan and J. Zheng.
\newblock The focusing cubic NLS with inverse-square potential
in three space dimensions.
\newblock {\em  Differential Integral Equations }, 30 (3/4): 161--206, 2017.



\bibitem{LFF}
S.~{Le~Coz}, R.~Fukuizumi, G.~Fibich, B.~Ksherim and Y.~Sivan.
\newblock Instability of bound states of a nonlinear Schr\"odinger
  equation with a dirac potential.
\newblock {\em Phys. D}, 237:1103--1128, 2008.

\bibitem{MiaMurphyZheng2021}
C.~Miao, J.~Murphy and J.~Zheng.
\newblock Threshold scattering for the focusing {NLS} with a repulsive
  potential.
\newblock {\em Preprint arXiv:2102.07163}, 2021.

\bibitem{Mizutani2020}
H.~Mizutani.
\newblock Wave operators on Sobolev spaces.
\newblock {\em Proc. Amer. Math. Soc.}, 148 (4): 1645--1652, 2020.


\bibitem{KSN}
K.~Schm\"udgen.
\newblock {\em Unbounded Self-adjoint Operators on Hilbert Space}, volume 265
  of {\em Graduate Texts in Mathematics}.
\newblock Springer, Dordrecht, 2012.

\bibitem{Weinstein1983}
M.~I. Weinstein.
\newblock Nonlinear Schr\"odinger equations and sharp interpolation
  estimates.
\newblock {\em Comm. Math. Phys.}, 87(4):567--576, 1983.

\end{thebibliography}

\end{document}